\documentclass[a4paper,11pt]{amsproc}
\usepackage{graphicx}
\usepackage{amscd}
\usepackage{color}

\theoremstyle{plain}

\newtheorem{lemma}{Lemma}

\newtheorem{proposition}{Proposition}

\newtheorem{theorem}{Theorem}
\numberwithin{equation}{section}

\definecolor{orange}{rgb}{1,0.5,0}
\definecolor{auburn}{rgb}{0.43, 0.21, 0.1}
\definecolor{amethyst}{rgb}{0.6, 0.4, 0.8}
\definecolor{amber}{rgb}{1.0, 0.49, 0.0}
\definecolor{brown}{rgb}{0.59, 0.29, 0.0}

\newcommand{\bm}{\boldmath}

\newcommand{\m}{\mbox}

\newcommand{\ubm}{\unboldmath}

\newcommand{\gr}{\m{\bm$\nabla$\ubm}}
\newcommand{\x}{\m{\bm$x$\ubm}}
\newcommand{\bu}{\m{\bm$u$\ubm}}
\newcommand{\bv}{\m{\bm$v$\ubm}}
\newcommand{\bw}{\m{\bm$w$\ubm}}

\newcommand{\f}{\m{\bm$f$\ubm}}

\newcommand{\bnu}{\m{\bm$\nu$\ubm}}
\newcommand{\btau}{\m{\bm$\tau$\ubm}}

\newcommand{\p}{\m{\bm$\phi$\ubm}}

\newcommand{\by}{\m{\bm$y$\ubm}}

\newcommand{\bz}{\m{\bm$z$\ubm}}
\newcommand{\be}{\m{\bm$e$\ubm}}

\begin{document}

\title[Long time behavior for a dissipative shallow water model]
{Long time behavior for a dissipative shallow water model}
\author[V. Sciacca, M.E. Schonbek and M. Sammmartino]{V. Sciacca$^1$, M.E. Schonbek$^2$ and M. Sammmartino$^3$}
\address{$^{1,3}$Dipartimento di Matematica, Universit\`a di Palermo,  90123 Palermo, Italy.}
\address{$^2$Department of Mathematics, UC Santa Cruz, Santa Cruz, CA 95064, USA}

\bigskip

\email{$^1$vincenzo.sciacca@unipa.it}
\email{$^2$schonbek@math.ucsc.edu}
\email{$^3$marco@math.unipa.it}

\maketitle

\begin{abstract}
We consider the two-dimensional shallow water model derived in ~\cite{ML01}, describing 
the motion of an incompressible fluid, confined
in a shallow basin, with varying bottom topography. We construct the
approximate inertial manifolds for the associated dynamical system and
estimate its order. Finally, considering the whole domain $\mathbb{R}^2$ and
under suitable conditions on the time dependent forcing term, we prove
the $L_2$ asymptotic decay of the weak solutions.

\end{abstract}

\section{Introduction}\label{sec:1}

In ~\cite{ML01}, the authors derived the following shallow water model:
\begin{eqnarray}\label{equa}
\frac{\partial \bu}{\partial t} +\bu \cdot \gr\bu +\gr p +\eta \bu =  \qquad \qquad 
\qquad \qquad \qquad \qquad && \nonumber \\
\qquad\qquad =b^{-1}\gr \cdot [b\nu (\gr\bu 
+(\gr\bu)^{T}-\mathbf{I}\gr \cdot\bu)] + \mathbf{f}, &&
\nonumber \\ 
\gr\cdot (b \bu) = 0, \qquad \qquad\quad\quad \qquad && \nonumber \\ 
\bu(\x,t=0)= \bu_0 ,\qquad \qquad \quad \quad \quad && \\ 
\bnu \cdot \bu = 0 \qquad \qquad \quad \x \in \partial\Omega, &&
\nonumber \\ 
\btau \cdot (\gr\bu +(\gr\bu)^{T} )\cdot \bnu =-\beta \bu \cdot \btau 
\qquad \x \in\partial\Omega. &&\nonumber
\end{eqnarray}
In the above system 
$\Omega \subset \mathbb{R}^{2}$ 
is a bounded domain with sufficiently regular boundary  $\partial 
\Omega$ and
$\bu(\x,t)$ denotes the velocity of the fluid at $\x \in \Omega$ and at time $t$. The smooth function $b(\x)$ describes the 
bottom topography and satisfies $0<b_i\leq b(\x) \leq b_s $, $\nu(\x)$ is the viscosity,
$\eta(\x)
$ is a positive smooth bounded function 
defined in  $\Omega$ representing the combined actions of the friction at 
the bottom and the wind pressure, $\mathbf{I}$ is the identity , 
$\btau$ and $\bnu$ are respectively the unity tangent and normal vector to the 
boundary $\partial\Omega$,
$\beta(z)$ is a regular function defined in $\partial\Omega$ giving the friction coefficient at the boundary, and $\f(\x)
$ is the 
force term which describes the wind stress.

System \eqref{equa}  was derived  in \cite{ML01} 
from a three-dimensional anisotropic eddy viscosity model of an 
incompressible fluid confined to a shallow basin with varying bottom topography.
To obtain the shallow water model
\eqref{equa}, the authors assumed that the depth of the basin is much smaller than
the typical horizontal length, and the typical velocity of the fluid is
much smaller than the velocity of the gravity waves. This last assumption
is equivalent to consider the fluid motion on time scales much longer than
the period of the gravity waves so that averaging on time suppresses gravity
waves.  The same assumptions had been used in \cite{CHL96} starting from the Euler equations
to derive the so called lake equations. The system \eqref{equa} 
is therefore a generalization of the lake equations as the effects of the viscous stresses 
are taken into account. In \cite{ML01} the well posedness of the model was also established. 

In this paper we construct approximate inertial manifolds whose order
decreases exponentially with respect to the dimension of the manifold. We
give the dependence of all the constants with respect to the corresponding physical parameters
and in particular we give explicitly the order of the approximate inertial
manifolds. 

When $\Omega =\mathbb{R}^2$ we address the problem of the asymptotic decay of the solutions.
Under suitable 
conditions on the forcing  term and of the initial datum, we show that the energy norm of weak solution has non-uniform decay .
A weak solution which satisfies a generalized energy inequality is constructed following ~\cite{Masuda84, ORS97, KPSark}. Then using the Fourier splitting method ~\cite{Sc85,Sc86,W87} non uniform $L^2$ decay is obtained.

Similar decay questions were originally proposed by Leray in ~\cite{Leray33,Leray34} for the Navier-Stokes equations. The first proof for decay without a rate was given by  Masuda in \cite{Masuda84}
and by Kato in ~\cite{Kato84} in the case of null force and strong solutions with small data.
Schonbek ~\cite{Sc85,Sc86}, using the Fourier Splitting Method, obtained the algebraic rate of decay for weak solution with large data.
See also ~\cite{BM92,H80,KM86,KO93,M92,U87}.

The plan of the paper is the following. 
In the next section, after introducing the appropriate mathematical settings for the model equations , we prove the 
existence of the Approximate Inertial Manifolds (AIM)  and,  then  give the thickness of the thin 
neighborhood in terms of the data.

In Section \ref{sec:unbounded}  we give the  preliminary results to establish the decay of the solutions.
In section \ref{sec:nu-decay} we prove the non-uniform asymptotic decay of the $L^2$ norm of the weak solution.

\section{Bounded domain: approximate inertial manifolds}

The concept of inertial manifold was introduced in ~\cite{FST88}, as part
of the theory of dissipative differential equations. An inertial manifold for a
semigroup associated to a dissipative dynamical system, is a finite dimensional 
Lipschitz manifold which is positively invariant, and attracts all the
orbits exponentially \cite{Ro01,YS02,Te97}.
To prove the existence of the inertial manifold it is necessary that the 
so called $spectral$ $gap$ condition ~\cite{Te97} is verified.
Unfortunately, this spectral gap condition is not verified for Navier-Stokes equations.
For this reason the
notion of approximate inertial manifolds (AIM)  was introduced 
~\cite{DD,DT,FMT88,FMRT01,Ro01,R95,SS04,Te97}. The 
existence of these manifolds does not require the spectral gap condition and 
therefore can be obtained for
a broader class of dissipative dynamic systems. The AIM can be defined as a
Lipschitz manifold surrounded by a thin neighborhood and each 
orbit of the system must enter in a finite time. The order of the manifold is 
the width of the thin neighborhood and is
exponentially small compared to the size of the AIM,
hence the AIM gives an approximation of the attractor of exponential order. 
The AIM theory plays an important role
in the development of new numerical algorithms suitable
to the approximation of dissipative systems for long times
~\cite{DDT95,FMT88,JMT,JX}.

In this Section we construct a sequence of approximate inertial manifolds  $\mathcal{M}_N$ for system \eqref{equa}.
Moreover, we show that the AIM $\mathcal{M}_N$ approximate the global attractor exponentially.
For the  proof of the existence of  $\mathcal{M}_N$ and to estimate the semidistance of the attractor to $\mathcal{M}_N$, 
we shall follow the ideas of \cite{DD,DT,R95,Te97}.

\subsection{The mathematical setting}\label{sec:bounded}
In this Section we shall briefly introduce the mathematical setting appropriate for 
\eqref{equa}. More details can be found in \cite{ML01}.  
One introduces the following Hilbert spaces: 
\begin{equation}
H=\{\bu : \bu \in L_{b}^{2},\,
\gr\cdot(b\bu)=0, \, \bnu \cdot\bu=0 \, \x \in
\partial\Omega \}
\end{equation}
\begin{equation}
V=\{\bu  : \bu \in H_{b}^{1},\,
\gr\cdot(b\bu)=0, \, \bnu \cdot\bu=0 \, \x \in
\partial\Omega \}
\end{equation}
where $L_{b}^{2}$
and $H_{b}^{1}$ are Sobolev spaces with scalar products and
weighted norms defined as:
\begin{eqnarray*}
	&&(\bu ,\bv )_{b}= \int_{\Omega} b \bu \cdot \bv d\x \, , \quad
	|\bu|^2_{b}= \int_{\Omega} b|\bu |^{2} d\x \, , \\
	&&  ((\bu ,\bv ))_{b}= 
	\int_{\Omega} b \gr \bu : \gr \bv d\x \, , \quad \|\bu\|^2_{b}= 
	\int_{\Omega} b|\gr \bu |^{2} d\x  \, .
\end{eqnarray*}
The following Poincar\'e inequality holds:
\begin{equation}\label{poincare}
|\bu|_b \leq \Pi \|\bu\|_b, 
\end{equation}
where $\Pi=\Pi(\Omega)$. 

{
	We take the $L^2_b$ scalar product of equation \eqref{equa} with a generic function $\bv \in V$ and write \eqref{equa} in the following weak form (see ~\cite{Te3}):
	\begin{equation}\label{eq_1}
	\frac{d}{dt}(\bu,\bv)_b+\left[\bu,\bv\right]_{b\nu}+(\bu,\bu,\bv)_b+ (\eta \bu,\bv)_b=(\f,\bv)_b,
	\end{equation}
	where $\left[\cdot, \cdot\right]_{b\nu}: V \times V \rightarrow \mathbb{R}$, is a bilinear form defined as  
	\begin{eqnarray}\label{eq_1_linear}
	\left[\bu,\bv\right]_{b\nu}&=&\int_{\Omega} b\nu \left(\gr\bu 
	+(\gr\bu)^{T}-\mathbf{I}\gr \cdot\bu \right)
	: \left(\gr\bv 
	+(\gr\bv)^{T}-\mathbf{I}\gr \cdot\bv \right)  d \x + \nonumber\\
	&&\qquad \qquad + \int_{\partial \Omega} b\nu \beta \bu \cdot \bv d s,
	\end{eqnarray}
	and, $(\cdot,\cdot,\cdot)_b:V\times V\times V \rightarrow \mathbb{R}$, is a trilinear form defined by
	\begin{equation}\label{eq_1_nonlinear}
	(\bu,\bw,\bv)_b=\int_{\Omega} b \left(\bu \cdot \gr\bw \right)\bv d \x.
	\end{equation}
	\\
	The trilinear form defines a continuous  
	bilinear operator $B(\bu, \bv)=\bu \cdot \gr \bv$ from $V\times V$ into $V^{'}$ such that
	\begin{equation}\label{def_B}
	(B(\bu,\bv),\bw)_b=(\bu, \bv , \bw)_b. 
	\end{equation}
	With $ A_{b \nu}$ we denote the operator from $V \rightarrow V^{'}$ defined by  
	\begin{equation}\label{def_Abnu}
	(A_{b\nu}\bu, \bv)_b = \left[\bu, \bv \right]_{b\nu}.
	\end{equation}
	We note that  $A_{b \nu}$ is a linear unbounded operator on $H$ with
	domain
	\begin{eqnarray}
	D(A_{b \nu})&=&\{\bu \in H^{2}_{b}(\Omega), \gr \cdot b\bu=0 \quad
	\text{in} \quad \Omega, \quad \bu \cdot \bnu =0, \nonumber\\
	&& \quad \btau \cdot (\gr \bu +(\gr \bu)^{T})\cdot \bnu =-\beta \bu \cdot \btau 
	\quad
	\text{on} \quad
	\partial \Omega \} ,\nonumber
	\end{eqnarray}
	and $D(A_{b \nu})\subset V\subset H\subset V^{'} $, where the inclusions
	are continuous and dense. Moreover
	$V$ is compactly embedded in $H$.\\
	We observe that $B(\bu, \bv): D(A_{b\nu })\times D(A_{b \nu
	}) \rightarrow H$ (see again ~\cite{Te3}).\\
	Using \eqref{def_B} and \eqref{def_Abnu}, we can write \eqref{eq_1}, the weak form of equation \eqref{equa}, as:
	\begin{equation}\label{eq}
	\frac{d}{dt}\bu+A_{b\nu }\bu+B(\bu,\bu)+ \eta \bu=\f.
	\end{equation}

	\noindent Note also that
	the bilinear form $\left[ \cdot, \cdot \right]_{b\nu}$ is 
	coercive, if $\beta(\x)\geq  \kappa (\x)$, where $\kappa$ is the
	curvature of $\partial\Omega$: Supposing this hypothesis on $\beta$ we have
	
	\begin{equation}\label{coe}(A_{b\nu }\bu , \bu)_{b}\geq \bar{b} \nu_i \|\bu\|^2_b,
	\end{equation}
	where
	$$
	\bar{b}= \frac{b_i}{b_s}\,, \quad  \nu_i=\inf_{\Omega} \nu(x) ,  
	$$
	and
	$$
	b_i=\inf_{\Omega} b(x)\,, \quad b_s=\sup_{\Omega} b(x) .
	$$
	For a proof of the  cohercivity  inequality \eqref{coe}  see \cite{ML01}.

	In \cite{ML01} the authors established the well-posedness of \eqref{eq}. For completeness we state their main result:
	
	\begin{theorem}  (theorem 4.1 of \cite{ML01} )Let $\Omega$ be smooth. Suppose that $b(\x)$, $\nu(\x)$ and $\eta(\x)$ are non negative function over $\bar{\Omega}$. Suppose, moreover that $b \nu \geq C >0$ and that $\beta(\x)\geq \kappa(\x)$ on $\partial \Omega$, where $\kappa(\x)$ is the curvature of $\partial \Omega$ at $\x$. Let $\bu_{in} \in H^2_b \cap V$  and $\f \in L^{2}_{b}$.\\
		Then the system \ref{equa} has a unique solution $\bu \in L^{\infty}\left([0,T],H^2_b\right)\cap C\left([0,T],V\right)$. Moreover, $\partial_t \bu \in L^{\infty}\left([0,T],H\right)\cup L^2\left([0,T],V\right)$.
	\end{theorem}
	\vskip0.5cm
	The spectral problem associated to the compact self-adjoint operator
	$A_{b\nu}$
	admits solution in $H$ ~\cite{Eva}, and from the coercivity
	\eqref{coe} derives the existence of a non-decreasing sequence of positive 
	eigenvalues $\{\lambda_n\}_{n\in\mathbb{N}}$  with (see
	~\cite{Met})
	\begin{equation}\label{asym_autovalore}
	\lambda_n \sim n \, , \qquad\mbox{for} \quad  n\rightarrow \infty \, , 
	\end{equation}
	and a sequence of eigenfunctions 
	forming an orthonormal basis in $H$.
	We denote by $P_n$ the projection onto the finite dimensional space generated 
	by the first
	$n$ eigenfunctions and $Q_n =I-P_n$:
	\begin{equation}\label{oper_proj}
	P_n \bu=\by, \qquad Q_n \bu =\bz \qquad \text{and} \quad 
	\bu=\by+\bz.  
	\end{equation}
	It is easy to verify that ~\cite{Henry}:
	\begin{eqnarray}
	&&|e^{- A_{b\nu}t} Q_n |_{\mathcal{L}(H,V)} \leq 
	\bar{b}^{-\frac{1}{2}} \left(
	(\nu_i t)^{-\frac{1}{2}} + \lambda_{n+1}^{\frac{1}{2}} \right)
	e^{- \lambda_{n+1}
		t}, \quad t> 0  ,\label{hp_2}\\
	&&  |(I+\tau A_{b\nu})P_n|_{\mathcal{L}(V)} \leq (1+ \tau \lambda_n)
	\leq e^{\tau \lambda_n}, \label{hp_4}\\
	&& |I|_{\mathcal{L}(P_n H, P_n V) }\leq 
	\left( \frac{\bar{b} \nu_i}{\lambda_{n}} \right)^{-\frac{1}{2}} ,\label{hp_5}
	\end{eqnarray}

	If we consider an initial datum $\bu_{in}$ in a ball of $H$ with center  at the 
	origin and radius $R$, then there exists a time $t_{0}(R)$, depending on $R$ and on $\nu , \f
	, \Lambda , b $, such that for $t \geq t_0 $:
	\begin{equation}\label{assorbHV}
	|\bu(t)|_b \leq \rho_0 , \qquad \|\bu(t)\|_b \leq \rho_1,
	\end{equation}
	where $\rho_0 $ and $\rho_1$ are the radii of the 
	absorbing balls in $H$ and $V$, respectively, whose explicit expressions is given in ~\cite{Ott04,SS01}.
	
	Moreover, it is possible to prove that:
	\begin{equation}\label{deriHV}
	\left|\frac{d^k \bu}{dt^k}\right|_b \leq \frac{2^k k!}{\alpha^k}\rho_0, \qquad
	\left\|\frac{d^k \bu}{dt^k}\right\|_b \leq \frac{2^k k!}{\alpha^k}\rho_1 ,
	\end{equation}
	for $t\geq 2 \alpha$, where $\alpha=\alpha \left( \Omega,| \mathbf{f} 
	|_b,\|\bu_{in}\|_b,\nu_i \right)$ define the domain of time analyticity 
	\begin{equation}
	\Delta=\{ \xi \in \mathbb{C}: \Re{\xi}\leq \alpha  \,  \text{and} 
	\, |\Im{\xi}|\leq \Re{\xi} \quad \text{or} \quad \Re{\xi}\geq \alpha \, 
	\text{and} \, 
	|\Im{\xi}|\leq \alpha \}. 
	\end{equation}

	From \eqref{assorbHV}, following ~\cite{Te97}, one can derive 
	the existence of a compact global attractor $\mathcal{A}$, connected and
	maximal in $H$, and its Hausdorff dimension $\tilde{m}$ satisfies the following estimate (see ~\cite{Ott04,SS01})
	$$ \tilde{m} -1\leq
	\frac{\bar{b}}{b_s^{1/2}}\widetilde{c}_\Omega \frac{| \f |_b
		\Pi^{1/2}}{\nu_i^2 } < \tilde{m} .$$
	
	For completeness,  we recall that that a global attractor $\mathcal{A}$ for a semigroup $S(t)$ defined in $H$, is a subset of $H$ which satisfies the following properties:
	\begin{itemize}
		\item $\mathcal{A}$ is an invariant set, i.e. $S(t)\mathcal{A}=\mathcal{A}$ for every $t\geq 0$,
		\item for every $\bu_0\in H$, it holds that 
		\begin{equation}
		\textrm{dist}\left(S(t)\bu_0,\mathcal{A}\right):=\inf_{\bv \in \mathcal{A}} \left|S(t)\bu_0 - \bv \right|_{b}\rightarrow 0, \quad \text{as} \quad t\rightarrow +\infty. \nonumber
		\end{equation} 
	\end{itemize}

	As it is usual in the theory of inertial manifold, we consider the associated
	equation  derived from \eqref{equa}, setting the non linear term $B(\bu,\bu)$ identically zero when $\bu$ is outside the absorbing ball in $V$. Specifically, let  $\theta \in \mathcal{C}^1$ be  defined on $\mathbb{R}_{+}$ 
	which is $1$ in $[0,1]$ and $0$ in $[2,+\infty[$.
	Denote by
	\begin{equation}\label{btheta}
	B_{\theta}\bu=B_{\theta}(\bu, \bu)=\theta\left( \frac{\|\bu
		\|_{b}^{2}}{\rho_{1}^{2}} \right)B(\bu, \bu).
	\end{equation}
	In the sequel consider the system:
	\begin{equation}\label{equa2}
	\frac{d \bu}{dt} +A_{b\nu}\bu +\eta \bu=B_{\theta}\bu +\f .
	\end{equation}
	Clearly, \eqref{equa2} is well posed, and 
	has the same attractor as \eqref{equa}. 
	Moreover, there 
	exist two constants $M_0$ and $M_1$ such that, for every $\bu, \bv \in V$, 
	\begin{equation}\label{hp12}
	|B_{\theta}\bu|_{b}\leq M_0 , \qquad
	|B_{\theta}\bu - B_{\theta}\bv|_{b}\leq M_1 \|\bu - \bv \|_{b}.
	\end{equation}


	\subsection{Existence of Approximate Inertial Manifolds}\label{sec:existenceAIM}
	
	An Inertial Manifold (IM) $\mathcal{M}=\{\by,\Phi(\by) \}$ is a positively invariant 
	manifold defined as the graph of a Lipschitz function $\Phi$, defined from $P_n 
	H$ to $Q_n H$, which attracts all trajectories of \eqref{equa2} exponentially. 
	
	\noindent We briefly outline the Lyapunov Perron Method  (\cite{DD,DT,FMT88,FMRT01,Ro01,R95,SS04,Te97}) which will be used in our proof to construct an IM.}

We decompose equation \eqref{equa2} using the projections $P_n$ and $Q_n$ to obtain a solution in $\mathcal{M}$

\begin{eqnarray}
\frac{d\by}{dt}+A_{b\nu}\by +\eta \by &=&P_n B_{\theta}(\by+\Phi(\by))+P_n \f 
\label{sistinerz}\\
\frac{d\Phi(\by)}{dt}+ A_{b\nu}\Phi(\by) +\eta \Phi(\by)&=&Q_n
B_{\theta}(\by+\Phi(\by))+Q_n \f .\label{sistinerz2}
\end{eqnarray}
The finite dimensional system of ordinary 
differential equations \eqref{sistinerz} is called the inertial system 
associated to $\mathcal{M}$.  Given the initial condition  $\bu_{in}=\by_{in}+\Phi(\by_{in})$, since $\Phi$ is a Lipschitz function, then  for every
$t\in \mathbb{R}$, the equation  \eqref{sistinerz} determines    a unique $\by(t)=\by(t;\by_{in},\Phi)$.


Assuming that $\Phi$ is bounded, to determine the function $\Phi$,  
integrate system \eqref{sistinerz2} in time to obtain 

\begin{equation}
\Phi (\by_{in})=\int^{0}_{-\infty}e^{A_{b\nu} 
	s}[Q_{n}\left(B_{\theta}(\by(s)+\Phi (\by(s)))  +\f \right)-\eta \Phi(\by(s))]ds .
\end{equation}

The function $\Phi$ is the fixed point of 
the map $\phi \rightarrow \mathcal{F}\phi$ defined by
\begin{equation}\label{map2}
\mathcal{F}\phi (\by_{in})=\int^{0}_{-\infty}e^{A_{b\nu} 
	s}[Q_{n}\left(B_{\theta}(\by(s)+\phi (\by(s)))  +\f \right)-\eta \phi(\by(s))]ds ,
\end{equation}
where $\phi:P_n H \rightarrow Q_n H$ is a bounded Lipschitz function. The existence of an inertial manifold is achieved  by showing 
that the map $ \mathcal{F}$ is a contractive  map in the complete metric space
\begin{equation}\label{spa}
\mathcal{F}_{l,L}=\{\phi:P_nV \rightarrow Q_n V : \, \mathrm{Lip}
(\phi) \leq l, \, |\phi|_{\infty}=\sup_{\by \in P_n V}\|\phi (\by)\|_{b}\leq 
L \},
\end{equation}
and $\by$ is a solution of the system \eqref{sistinerz} with 
$\by(t=0)=\by_{in}$.
%
We recall that the proof of the existence of the Inertial Manifold 
is based on the spectral gap condition.

If the spectral gap condition is not verified, there is no standard  proof for the existence of the IM. 
However, it is possible to construct a 
sequence of Approximate Inertial Manifolds \cite{DD,DT,FMT88,FMRT01,Ro01,R95,SS04,Te97}. This is what we will do for equation \eqref{eq}. 

To obtain  the AIM we construct an   approximating sequence  of  solutions to  the system \eqref{sistinerz} as follows: 
let $\by_0  =\by_{in} \in P_{n}V$, and  $\tau>0$ be the discrete time step and define $\by_k$, $k\geq 0$, by the following Euler explicit discretization of \eqref{sistinerz}:
\begin{equation}\label{discr}
\frac{\by_{k+1}-\by_{k}}{-\tau}+A_{b\nu} 
\by_{k}=P_{n}B_{\theta}(\by_{k}+\phi(\by_{k}))-\eta \by_k+P_{n}\f,
\end{equation}
Fix  the positive integers $n$ and  $N$, to  construct 
the approximation  function $\by_{\tau}$ to $\by$:
\begin{eqnarray}
\quad \by_{\tau}(s)&=&\by_k \quad \mathrm{for} \quad -(k+1)\tau<s\leq
-k\tau ,\label{stepfunc1} \quad k=0,\dots,N-1 , 
\nonumber \\
\quad \by_{\tau}(s)&=&\by_N \quad \mathrm{for} \quad s\leq-N\tau .
\label{stepfunc2}
\end{eqnarray} 


The approximation $ \mathcal{F}^{N}_{\tau}$ of $\mathcal{F}$ is defined 
substituting $\by$ by  $\by_{\tau}$ in \eqref{map2}. Explicitly

\begin{eqnarray}
&&\mathcal{F}^{N}_{\tau}\phi (\by_{0})= \nonumber \\
&&-(A_{b\nu})^{-1}(I-e^{- A_{b\nu}})\sum_{k=0}^{N-1} e^{-k A_{b\nu} \tau} 
[Q_{n} \left(B_{\theta}(\by_{k} + \phi(\by_{k} ))+\f\right)-\eta 
\phi(\by_{k})]\nonumber
\\&&\qquad -(A_{b\nu})^{-1} e^{- N A_{b\nu} 
	\tau}[Q_{n}\left(B_{\theta}(\by_N 
+ \phi(\by_N
))+\f\right)-\eta \phi( \by_{N})]. \label{mapdiscr}
\end{eqnarray}
To obtain the family of AIM,  consider a sequence of positive numbers $\left( \tau_N \right)_{N\in \mathbb{N}}$ and 
define the manifolds $ \mathcal{M}_N$ as the graph of the functions $\Phi_N$ 
constructed recursively, for $N\geq 0$, by
\begin{equation}
\Phi_0 =0 , \quad \Phi_{N+1}=\mathcal{F}_{\tau_N}^{N}(\Phi_{N}). 
\label{costruz}
\end{equation} 
The main result of this section is to prove, for every $N\geq 0$, the existence of
$\Phi_N$ in $\mathcal{F}_{l,L}$. f
Before proceeding with the formulation of the main theorem of this section and 
its proof, we recall some preliminary properties, which guaranties the 
consistence of the approximation scheme described above.
To ease the notation in the sequel, we  denote $\tau_N$ by $\tau$. 
Write \eqref{discr} as
\begin{eqnarray}\label{richi1}
\qquad \qquad \by(-(k+1)\tau)&=&(I+ \tau A_{b\nu})\by (-k\tau)   \\
&& -\tau P_{n}\left(B_{\theta}(\by(-k\tau) +\bz(-k\tau))) 
+\f\right)+\tau \eta \by(-k\tau),\nonumber
\end{eqnarray}
and the approximation error 
\begin{equation}\label{approx_error}
\epsilon_k =\by(-(k+1)\tau)-\by(-k\tau)-\tau\frac{d\by}{dt}(-k\tau).
\end{equation}
The following Lemmas hold:
\begin{lemma}
	Suppose that $\bu(t)$ is a complete trajectory inside the global attractor 
	$\mathcal{A}$, then:
	\begin{equation}\label{hp7_1}
	\| \epsilon_k\|_{b}\leq \tau^{2} \beta_1 , \qquad k=0,\dots,N-1,
	\end{equation}
	\begin{equation}\label{hp7_2}
	\left\| \frac{d \bu}{dt} \right\|_{b}\leq  \beta_2 , \qquad t<0,
	\end{equation}
	with $\beta_1 \leq \frac{8}{\alpha^2}\rho_1 $ and $\beta_2\leq \frac{2 
		\rho_1}{\alpha}$, where $\alpha$ defines the domain of analyticity in \eqref{deriHV} and $\rho_1$ is the radius of the absorbing balls in $V$ in \eqref{assorbHV}.
\end{lemma}
\begin{proof}
	If the trajectory $\bu(t)$ is a complete trajectory inside the global attractor 
	$\mathcal{A}$, one can easy obtain \eqref{hp7_1} and \eqref{hp7_2}
	with 
	$$\beta_1= 
	\sup_{[\bu(t)]_{t\in\mathbb{R}}\in \mathcal{A}}
	\sup_{t\in\mathbb{R}}\left\|\frac{d^2 \bu}{dt^2}\right\|_b \; \qquad 
	\beta_2= 
	\sup_{[\bu(t)]_{t\in\mathbb{R}}\in
		\mathcal{A}} \sup_{t\in\mathbb{R}}\left\|\frac{d \bu}{dt}\right\|_b$$
	Using \eqref{deriHV} one derives the desired bounds on $\beta_1$ and $\beta_2$.
	  \end{proof}

\begin{lemma}
	Let be $i=1,2$, and let be $\by^{i}_{0}\in P_{n}V$. Define $\by^{i}_{k}$,
	$k=0,\dots,N$ by \eqref{discr} and \eqref{richi1} with $\by_0 =\by^{i}_{0}$
	and construct $\by^{i}_{\tau}(s)$ using
	\eqref{stepfunc1}. Then, for every $s\leq 0$,
	\begin{equation}\label{le2}
	\| \by^{1}_{\tau}(s)-\by^{2}_{\tau}(s) \|_{b}\leq e^{-s  [ \lambda_{n} + 
		\left( \frac{\bar{b} \nu_i}{\lambda_{n}} \right)^{-\frac{1}{2}} (M_1 +\Pi \bar{\eta}) (1+l)  ] } \| 
	\by_0^1 
	-\by_0^2 \|_b ,
	\end{equation}
	where 
	$\bar{\eta}=\sup_{\Omega} \eta$, and $M_1$ is given in \eqref{hp12}.
\end{lemma}
\begin{proof}
	Denoting by $\by_k =\by^{1}_{k}-\by^{2}_{k}$ and subtracting
	\eqref{discr} or \eqref{richi1} for $i=1,2$, 
	and
	using \eqref{hp12}, \eqref{hp_4} and the
	Lipschitz constant $l$ of $\phi$, we obtain
	\begin{eqnarray*}
		&& \| \by_{k+1} \|_b \leq (1+\tau \lambda_{n} ) \|\by_k \|_b + \tau 
		\left( \frac{\bar{b} \nu_i}{\lambda_{n}} \right)^{-\frac{1}{2}}
		\cdot \\ 
		&& \quad \cdot
		\left[ 
		|B_{\theta}(\by^{1}_{k}+ \phi(\by_{k}^{1}))-B_{\theta}(\by^{2}_{k}+ 
		\phi(\by_{k}^{2}))|_b + 
		\bar{\eta} |\by_{k}+ \phi(\by_{k}) |_b 
		\right]\\
		&& \leq (1+\tau \lambda_{n} ) \|\by_k \|_b +\tau
		\left( \frac{\bar{b} \nu_i}{\lambda_{n}} \right)^{-\frac{1}{2}} (M_1 +\Pi 
		\bar{\eta}) 
		(1+l) \|\by_k \|_b \\
		&& \leq \exp\{ k \tau [ \lambda_{n} + 
		\left( \frac{\bar{b} \nu_i}{\lambda_{n}} \right)^{-\frac{1}{2}} (M_1 +\Pi \bar{\eta}) (1+l)] \} 
		\| 
		\by_0 \|_b .
	\end{eqnarray*}
	for $k=0, \dots ,N$. From the definition of
	$\by_{\tau}^{i}(s)$ by \eqref{stepfunc1}, we obtain \eqref{le2}.
	  \end{proof}

In the sequel we use the notation
$\gamma=\int^{0}_{-\infty}|s|^{-1/2}e^{s}d s$. We are now ready to establish  the main theorem.
\begin{theorem}\label{primo}
	Suppose that the constants $\delta_1$ and $\delta_2$ satisfy
	\begin{equation}\label{hpteo1}
	(N+1) \tau \leq \frac{\delta_1}{ (M_1 +\Pi \bar{\eta})}
	\left( \frac{\bar{b} \nu_i}{\lambda_{n}}\right)^{\frac{1}{2}} ,
	\end{equation}
	and
	\begin{equation}\label{hpteo2}
	\lambda_{n} \geq \delta_2,
	\end{equation}
	then there exist $l$ and $L_0$ such that
	$\mathcal{F}_{\tau}^{N}:\mathcal{F}_{l,L}\rightarrow 
	\mathcal{F}_{l,L}$, for all $L\geq L_0$.
\end{theorem}
\begin{proof}
	We show that the following constant $L_0$ and $l$ are appropriate
	\begin{equation}\label{L_0} L_0 = \bar{b}^{-1/2}(|\f|_b+M_0 +\bar{\eta} \rho_0)(\gamma 
	\nu_i^{-1/2}+1)\lambda_{n+1}^{-1/2}.\end{equation}
	\begin{equation}\label{l_lip}
	l=6 \left(\frac{1}{2}+  \sup_n \left(\frac{\nu_i\lambda_{n+1}}{\lambda_{n}}\right)^{\frac{1}{2}} \right) \quad \text{and} \quad
	\delta_1=\min \left(\delta_0, \frac{\log\left(3/2 \right)}{l} \right).
	\end{equation}
	Let $\phi \in \mathcal{F}_{l,L}$, suppose $\by_{0}\in P_{n}V$ and
	$(\by_k )_{k=0,\dots,N}$ and $\by_\tau (s)$ be given by
	\eqref{discr}, \eqref{stepfunc1}.

	Using \eqref{hp12},  \eqref{hp_2} and recalling that  by definition $\phi(y) \in Q_n V \subseteq V$, we 
	have:
	\begin{eqnarray}
	&& \qquad \left\| \mathcal{F}^{N}_{\tau}\phi (\by_{0}) \right\|_b \leq \\
	&&\leq \int^{0}_{-\infty}\left|e^{ A_{b\nu}s}Q_{n}\right|_{\mathcal{L}(H,V)}
	|\f+B_{\theta}(\by(s)+\phi (\by(s))) -\eta (\by(s)+\phi (\by(s)))|_b d s
	\nonumber\\
	&&\leq\bar{b}^{-1/2}(|\f|_b+M_0 +\bar{\eta} \rho_0 
	)\int^{0}_{-\infty}\left( |\nu_i s|^{-1/2}
	+\lambda_{n+1}^{1/2} \right)e^{\lambda_{n+1} s} d s \nonumber\\
	&&\leq \bar{b}^{-1/2}(|\f|_b+M_0 +\bar{\eta} \rho_0)(\gamma 
	\nu_i^{-1/2}+1)\lambda_{n+1}^{-1/2}. \nonumber
	\end{eqnarray}

	From the previous inequality we deduce that $\left\| 
	\mathcal{F}^{N}_{\tau}\phi (\by_{0}) \right\|_b \leq L$, for every $L\geq L_0$ , where $L_0$ was defined by \eqref{L_0}.
	Now, we  show that  $l$ is our Lipschitz constant.
	For this scope, let $\by^i_0 \in P_{n}V$ and $(\by^i_k )_{k=0,\dots,N}$ and
	$\by^i_\tau (s)$ constructed by \eqref{discr} and
	\eqref{stepfunc1}, for $i=1,2$. Therefore, write

	\begin{eqnarray}
	&&\mathcal{F}^{N}_{\tau}\phi (\by^1_{0})-\mathcal{F}^{N}_{\tau}\phi
	(\by^2_{0})=\int^{0}_{-(N+1)\tau}e^{A_{b\nu}s}[ Q_{n} ( 
	B_{\theta}(\by^1_{\tau}(s)+\phi (\by^1_{\tau}(s)))
	\nonumber
	\\ &&\quad-B_{\theta}(\by^2_{\tau}(s)+\phi (\by^2_{\tau}(s))) ) -\eta (\phi 
	(\by^1_{\tau}(s))- \phi (\by^2_{\tau}(s)) )] d s
	\nonumber\\
	&&\qquad +(A_{b\nu})^{-1} e^{-(N+1)A_{b\nu} \tau} [Q_{n}(B_{\theta}(\by^1_N +
	\phi(\by^1_N ))\nonumber \\ &&\qquad\qquad- B_{\theta}(\by^2_N +
	\phi(\by^2_N ))) -\eta (\phi (\by^1_{N})- \phi (\by^2_{N}) )].
	\end{eqnarray}
	Using again \eqref{hp12}, \eqref{hp12} and \eqref{hp_2}, we have:
	\begin{eqnarray}
	&&\left\| \mathcal{F}^{N}_{\tau}\phi(\by^1_{0})
	-\mathcal{F}^{N}_{\tau}\phi (\by^2_{0}) \right\|_b \leq  \bar{b}^{-\frac{1}{2}} ( M_1 +\Pi \bar{\eta}) (l+1) \cdot \nonumber \\ 
	&& \qquad \cdot \int^{0}_{-(N+1)\tau}\left( |\nu_i s|^{-1/2} 
	+\lambda_{n+1}^{1/2} \right) 
	e^{\lambda_{n+1} s} \| \by^1_{\tau}(s)-\by^2_{\tau}(s)\|_b
	d s \nonumber\\
	&& + \bar{b}^{-\frac{1}{2}} ( M_1 +\Pi\bar{\eta}) (l+1)
	\nu_i^{-1/2}\lambda_{n+1}^{-1/2}e^{-\lambda_{n+1} (N+1)\tau}\|
	\by^1_{N}-\by^2_{N}\|_b \nonumber .
	\end{eqnarray}
	Using \eqref{le2}, since $\lambda_{n+1} -\lambda_{n} \geq0$, using \eqref{hpteo1}, we obtain 
	\begin{equation}
	\left\| \mathcal{F}^{N}_{\tau}\phi
	(\by^1_{0})-\mathcal{F}^{N}_{\tau}\phi (\by^2_{0}) \right\|_b \leq
	\Xi \|\by_0^1 -\by_0^2 \|_b, 
	\end{equation}
	with
	\begin{eqnarray}
	&& \Xi = \bar{b}^{-\frac{1}{2}} (l+1) e^{\delta_1 (l+1)} 
	\left[ 2\left(M_1+\Pi\bar{\eta}\right)^{\frac{3}{2}}
	\left(\frac{\bar{b}}{\nu_i\lambda_{n}}\right)^{\frac{1}{4}} +\right. 
	\label{hpteo3} 
	\\
	&& \qquad \qquad \left.+ (M_1+\Pi \bar{\eta})^2 (\nu_i 
	\lambda_{n+1})^{-\frac{1}{2}}\right]+  e^{\delta_1
		(l+1)}\left(\frac{\nu_i\lambda_{n+1}}{ \lambda_{n}}\right)^{\frac{1}{2}} . 
	\nonumber
	\end{eqnarray}
	We now choose $\delta_1$ and $\delta_2$ to ensure that $\Xi\leq l$ then the proof of the theorem will be complete.
	First choose $\delta_0 >0$, with $\delta_1 \leq \delta_0$, and  choose $\delta_2$ in \eqref{hpteo2} sufficiently large so that
	\begin{eqnarray}
	&&\bar{b}^{-\frac{1}{2}}\left[ 2\left(M_1+\Pi\bar{\eta}\right)^{\frac{3}{2}}
	\left(\frac{\bar{b}}{\nu_i\lambda_{n}}\right)^{\frac{1}{4}} 
	+ (M_1+\Pi \bar{\eta})^2 (\nu_i 
	\lambda_{n+1})^{-\frac{1}{2}}\right]  e^{\delta_1 } \leq \nonumber \\
	&& \quad \leq \bar{b}^{-\frac{1}{2}}\left[ 2\left(M_1+\Pi\bar{\eta}\right)^{\frac{3}{2}}
	\left(\frac{\bar{b}}{\nu_i\delta_2}\right)^{\frac{1}{4}} 
	+ (M_1+\Pi \bar{\eta})^2 (\nu_i 
	\delta_2)^{-\frac{1}{2}}\right]  e^{\delta_0 } \leq \nonumber\\ &&\quad \leq \frac{1}{2} .
	\end{eqnarray}
	Therefore, with this choice of $\delta_0$, $\delta_2$ and $\delta_1 \leq \delta_0$ we have
	\begin{equation}
	\Xi\leq  \left( \frac{l}{2} + \frac{1}{2}+  \sup_n \left(\frac{\nu_i\lambda_{n+1}}{\lambda_{n}}\right)^{\frac{1}{2}}  \right) e^{l \delta_1 } \leq l,
	\end{equation}
	by choosing $l$ as defined in \eqref{l_lip}  at the beginning of the Theorem. This completes the proof of the Theorem. 
	
\end{proof}



\subsection{Approximation of the attractor}\label{sec:approximationAIM}

In this section we prove that the approximate inertial manifolds 
$\mathcal{M}_N$ built
in the previous section as a graph of the $\Phi_N$, approximates
the global attractor $\mathcal{A}$. 

We first try to estimate the semi-distance in $V$ of $\mathcal{A}$ to 
$\mathcal{M}_N$
\begin{equation}\label{dist}
\varrho_N = d_V (\mathcal{A} , \mathcal{M}_N )= \sup_{\bv \in
	\mathcal{A}} \inf_{\bw \in \mathcal{M}_N } \|\bv -\bw \|_b.
\end{equation}
\noindent We continue to use the notations of the previous sections.
\begin{lemma} Let be $\bu_0 \in \mathcal{A}$ and let be $\by_0=P_n \bu_0$ and  $\bz_0=Q_n\bu_0$, with $P_n$ and $Q_n=I-P_n$ the projection operators defined in \eqref{oper_proj}.
	Suppose that \eqref{hp7_1}, \eqref{hp7_2}, \eqref{hpteo1} and  
	\eqref{hpteo2} are satisfied. Then for every $\phi \in 
	\mathcal{F}_{l,L}$ results that
	\begin{eqnarray}\label{spessore}
	&&\|\mathcal{F}_{N}^{\tau}\phi (\by_0 )-\bz_0 \|_b \leq \nonumber\\
	&& \leq \bar{b}^{-\frac{1}{2}} (M_1+\Pi \bar{\eta})  
	\left[ l (\lambda_{n}\nu_i)^{-\frac{1}{2}} + \frac{(\gamma 
		\nu_i^{-\frac{1}{2}} +1)}{\lambda_{n+1}^{\frac{1}{2}}} \right] 
	\sup_{\by +\bz \in \mathcal{A}}\|\phi(\by) - \bz\|_b 
	\nonumber \\
	&& \quad +\left[ \beta_1 l \lambda_{n}^{-1} + \beta_2 \bar{b}^{-\frac{1}{2}} (M_1+\Pi\bar{\eta}) (1+l) 
	\frac{(\gamma \nu_i^{-\frac{1}{2}} + 1)}{ \lambda_{n+1}^{\frac{1}{2}}} 
	\right] 
	\tau  
	\nonumber \\
	&&\quad + 2\bar{b}^{-\frac{1}{2}} 
	(M_0+\bar{\eta}\rho_0) 
	\frac{[\nu_i(N+1)\tau]^{-\frac{1}{2}}+\lambda_{n+1}^{\frac{1}{2}}}{\lambda_{
			n+1}} e^{
		-\lambda_{n+1}(N+1) \tau} .
	\end{eqnarray}
\end{lemma}
\begin{proof}
	Take $\phi \in \mathcal{F}_{l,L}$ and $\bu_0 = \by_0 +\bz_0 \in
	\mathcal{A}$ 
	a point in the global attractor. Denote with $(\bu (t)
	)_{t\in \mathbb{R}}$ the trajectory in $\mathcal{A}$ which pass through 
	$\bu_0$ at $t=0$. Consider $\by(t)=P_{n}\bu(t)$, $\bz(t)=Q_{n}\bu(t)$.
	Define $\widetilde{\by}_k =\by(-k\tau)$ and $(\by_k
	)_{k=0,\dots,N}$ with \eqref{discr}; and consider $\by_\tau$ constructed by 
	\eqref{stepfunc1}. 
	Using \eqref{hp12} and \eqref{hp_2}, the Lipschitz property of
	$\phi$, the Poincar\`e inequality \eqref{poincare} and \eqref{assorbHV}, we 
	have:
	\begin{eqnarray}\label{riciam}
	&&\|\mathcal{F}_{N}^{\tau}\phi (\by_0 )-\bz_0 \|_b \leq 
	\bar{b}^{-\frac{1}{2}} (M_1+\Pi \bar{\eta}) (1+l)
	\cdot \\
	&& \qquad \qquad \cdot 
	\int_{-(N+1)\tau}^{0}(|\nu_i 
	s|^{-\frac{1}{2}}+\lambda_{n+1}^{\frac{1}{2}})e^{\lambda_{n+1} s}  \| 
	\by_{\tau}(s) -\by(s) 
	\|_b ds\nonumber \\
	&& \qquad + (M_1+\Pi \bar{\eta}) \bar{b}^{-\frac{1}{2}} (\gamma 
	\nu_i^{-\frac{1}{2}} +1) \lambda_{n+1}^{-1}
	\sup_{\by +\bz \in \mathcal{A}} \|\phi(\by) - \bz\|_b  \nonumber \\
	&&\qquad + 2\bar{b}^{-\frac{1}{2}} (M_0+\bar{\eta} \rho_0)
	\frac{[\nu_i(N+1)\tau]^{-\frac{1}{2}}+\lambda_{n+1}^{\frac{1}{2}}}{\lambda_{
			n+1}} e^{
		-\lambda_{n+1}(N+1) \tau} .\nonumber
	\end{eqnarray}
	To estimate the integral on \eqref{riciam}, from \eqref{hp7_2}, for every\\ $s$ 
	in 
	$(-(k+1)\tau, -k\tau]$, we have
	\begin{equation}
	\| \by_\tau (s) -\by(s)\|_b  \leq \|\be_k \|_b 
	+ |k\tau +s | \sup_{\zeta \leq0}\|\frac{d\by}{dt}(\zeta)\|_b \leq  \|\be_k \|_b 
	+\tau \beta_2, \nonumber
	\end{equation}
	with $\be_k = \by_k-\widetilde{\by}_k$.
	Using \eqref{richi1} and \eqref{approx_error}, we have
	\begin{eqnarray}
	\|\be_{k+1}\|_b &\leq& (1+\tau \lambda_n ) \|\be_k \|_b + 
	\tau \left(\frac{\bar{b}  \nu_i}{\lambda_{n}}\right)^{-1/2} (M_1+\Pi \bar{\eta}) 
	(1+l) 
	\|\be_k \|_b
	\nonumber \\
	&& + \tau \left(\frac{\bar{b}  \nu_i}{\lambda_{n}}\right)^{-1/2} (M_1 +\Pi 
	\bar{\eta}) 
	\| \phi(\widetilde{\by}_k)-\bz(-k\tau)\|_b + \|\epsilon_k \|_b, \nonumber 
	\end{eqnarray}
	and from \eqref{hp7_1}, we have
	\begin{eqnarray}
	\|\be_{k}\|_b &\leq& [ (M_1 + \Pi \bar{\eta}) 
	\left( \bar{b}  \bar{\nu }\lambda_{n} \right)^{-\frac{1}{2}} 
	\sup_{\by +\bz \in \mathcal{A}} \| \phi(\by)-\bz\|_b + \tau \beta_1 
	\lambda_{n}^{-1}] \nonumber\\
	&&\quad \cdot   \exp \{ k \tau[\lambda_{n} +
	\left(\frac{\bar{b}\nu_i}{\lambda_{n}}\right)^{-\frac{1}{2}} (M_1+\Pi\bar{\eta}) 
	(1+l) 
	] \}.
	\nonumber
	\end{eqnarray}
	Now we are ready to estimate the first integral on \eqref{riciam}:
	\begin{eqnarray}
	&&\bar{b}^{-1/2} (M_1+\Pi \bar{\eta}) (1+l)  \cdot \nonumber \\
	&& \qquad \qquad \cdot 
	\int_{-(N+1)\tau}^{0}(|\nu_i 
	s|^{-1/2}+\lambda_{n+1}^{1/2})e^{\lambda_{n+1} s} 
	\| \by_{\tau}(s)-\by(s) \|_b ds \nonumber \\
	&& \leq 
	l [ (M_1+\Pi \bar{\eta})  
	\left( \bar{b}  \lambda_{n} \nu_i \right)^{-\frac{1}{2}}
	\sup_{\by +\bz \in \mathcal{A}} \| \phi(\by)-\bz\|_b + \tau \beta_1  
	\lambda_{n}^{-1}] \nonumber \\
	&& 
	\qquad 
	+ \tau \beta_2 \bar{b}^{-\frac{1}{2}} (M_1+ \Pi \bar{\eta}) 
	(1+l) (\gamma \nu_i^{-\frac{1}{2}}+ 1) \lambda_{n+1}^{-\frac{1}{2}} . 
	\nonumber
	\end{eqnarray}
	Combining the previous estimate with \eqref{riciam}, we have
	\begin{eqnarray}
	&&\|\mathcal{F}_{N}^{\tau}\phi (\by_0 )-\bz_0 \|_b \leq \nonumber \\
	&\leq& l[ (M_1+\Pi \bar{\eta}) 
	\left( \bar{b}  \nu_i \lambda_{n} \right)^{-\frac{1}{2}}
	\sup_{\by +\bz \in \mathcal{A}} \| \phi(\by)-\bz\|_b + \tau \beta_1 
	\lambda_{n}^{-1}] \nonumber \\
	&& \quad 
	+ \tau \beta_2 \bar{b}^{-\frac{1}{2}} (M_1+\Pi \bar{\eta}) 
	(1+l) (\gamma \nu_i^{-\frac{1}{2}}+ 1)
	\lambda_{n+1}^{-1/2}  \nonumber \\
	&& \quad 
	+ ( M_1+\Pi \bar{\eta}) 
	\bar{b}^{-\frac{1}{2}} (\gamma \nu_i^{-\frac{1}{2}} +1) 
	\lambda_{n+1}^{-\frac{1}{2}}
	\sup_{\by +\bz \in \mathcal{A}} \|\phi(\by) - \bz\|_b  \nonumber \\
	&&\quad + 2 \bar{b}^{-\frac{1}{2}} (M_0+\bar{\eta}\rho_0)
	\frac{[\nu_i(N+1)\tau]^{-\frac{1}{2}}+\lambda_{n+1}^{1/2}}{\lambda_{n+1}}e^
	{ - 
		\lambda_{n+1}(N+1) \tau} ,
	\end{eqnarray}
	which is 
	the \eqref{spessore}.
	  \end{proof}

In the next theorem we give 
an estimate on the number $n$ of modes to yield  an exponential 
approximation of $\mathcal{M}_N$ of the attractor, for $N$ large.

\begin{theorem}
	Suppose that the hypothesis \eqref{hpteo1} and \eqref{hpteo2} of Theorem 1 hold and that \eqref{hp7_1} and 
	\eqref{hp7_2} of Lemma 1 hold. 
	Assume, moreover, that the sequence $\tau_N$ satisfies 
	\begin{equation}
	\chi \leq \tau_N (N+1) \leq \frac{\delta_1}{(M_1+\Pi \bar{\eta})}\left(
	\frac{\bar{b} \nu_i}{\lambda_n} \right)^{1/2} , \label{teo2.2}
	\end{equation}
	for all  $N\in \mathbb{N}$; where $\chi$ is any fixed constant less then $\delta_1$.
	There exist a constant $\delta_3$ such that if $n$ is fixed by 
	\begin{equation}
	\lambda_n \geq \max\left(\delta_2,\delta_3\right), 
	\end{equation}
	then the approximate inertial manifolds 
	$\mathcal{M}_N$, constructed in Theorem 1, satisfy  
	\begin{equation}\label{teo2.4}
	d_V (\mathcal{A},\mathcal{M}_N) \leq  4 \bar{b}^{-1/2} 
	(M_0+\bar{\eta}\rho_0) \frac{1}{
		\lambda_{n+1}^{1/2}}e^{- \lambda_{n+1} \chi } ,
	\end{equation}
	for $N$ sufficiently large.
\end{theorem}

\begin{proof}
	
	Using the expression \eqref{spessore} in the previous Lemma, we have 
	$\varrho_{N+1} \leq \mu \varrho_N + \sigma_N $, where
	\begin{equation}
	\mu =\bar{b}^{\frac{1}{2}} (M_1+\Pi \bar{\eta}) \left[ l 
	(\lambda_n\nu_i)^{-\frac{1}{2}} + 
	(\gamma \nu_i^{-\frac{1}{2}} +1)\lambda_{n+1}^{-\frac{1}{2}} \right]  ,
	\end{equation}
	and
	\begin{eqnarray}
	\sigma_N &=& \tau_N \left[ \beta_1 l \lambda_n^{-1} + \beta_2 
	\bar{b}^{-\frac{1}{2}} (M_1+\Pi\bar{\eta}) (1+l)\frac{(\gamma 
		\nu_i^{-\frac{1}{2}} + 1)}{ \lambda_{n+1}^{\frac{1}{2}}} \right]  
	\nonumber \\
	&&\hskip-0.5cm+ 2\bar{b}^{-\frac{1}{2}} 
	(M_0+\bar{\eta}\rho_0) 
	\frac{[\nu_i(N+1)\tau_N]^{-\frac{1}{2}}+\lambda_{n+1}^{\frac{1}{2}}}{
		\lambda_{n+1}}e^{-\lambda_{n+1}(N+1) \tau_N} .
	\end{eqnarray}
	Iterating, we obtain $\varrho_N \leq \mu^N \varrho_0 + \sum_{0}^{N-1} 
	\sigma_{N-j-1} \mu^j$ , with $\varrho_0 = \sup_{\bu_0 \in 
		\mathcal{A}}\|\bz_0\|_b$. 
	By \eqref{assorbHV}, \eqref{hp12}, \eqref{hp_2}:
	\begin{equation}\label{r0}
	\|\bz\|_b \leq  
	\bar{b}^{-\frac{1}{2}} ( M_0 +|\f|_b +\bar{\eta}\rho_0)(\gamma 
	\nu_i^{-1/2}+1) \lambda_{n+1}^{-1/2} . 
	\end{equation}
	Using \eqref{teo2.2} 
	we obtain:
	\begin{eqnarray}
	&&\sum_{0}^{N-1} \sigma_{N-j-1} \xi^j = 2\bar{b}^{-\frac{1}{2}} 
	(M_0+\bar{\eta}\rho_0) \frac{1}{
		\lambda_{n+1}^{1/2}}e^{- \lambda_{n+1} \chi
	}\left(\sum_{0}^{N-1} \xi^j\right) \nonumber
	\\
	&&  +  \left[ \beta_1 l \lambda_n^{-1} + \beta_2 
	\bar{b}^{-\frac{1}{2}} (M_1+\Pi\bar{\eta}) (1+l)\frac{(\gamma 
		\nu_i^{-\frac{1}{2}} + 1)}{ \lambda_{n+1}^{\frac{1}{2}}} 
	\right] \left(\sum_{0}^{N-1}
	\tau_{N-j-1} \xi^{j} \right), \nonumber
	\end{eqnarray}
	for $N\geq (\tau_N \nu_i)^{-1} $ and supposing that $\mu \leq 
	\frac{1}{2}$, 
	we have
	\begin{eqnarray}\label{sum}
	&&\qquad \sum_{0}^{N-1} \sigma_{N-j-1} \mu^j = 4 
	\bar{b}^{-\frac{1}{2}} (M_0+\bar{\eta}\rho_0) 
	\frac{1}{\lambda_{n+1}^{1/2}}e^{-\lambda_{n+1} \chi } 
	\\
	&& + 2
	\left[ \beta_1 l \lambda_n^{-1} + \beta_2 \bar{b}^{-\frac{1}{2}} (M_1+\Pi\bar{\eta}) (1+l)\frac{(\gamma 
		\nu_i^{-\frac{1}{2}} + 1)}{ \lambda_{n+1}^{\frac{1}{2}}} \right]
	\sup_{0\leq j\leq
		N-1} \tau_{N-j-1} . \nonumber
	\end{eqnarray}
	Combining \eqref{r0} and \eqref{sum} we obtain that
	\begin{eqnarray}\label{teo2.3}
	&& d_V (\mathcal{A},\mathcal{M}_N) \leq 2^{-N}\bar{b}^{-\frac{1}{2}}
	( M_0 +\bar{\eta}\rho_0+|\f|_b)(\gamma \nu_i^{-1/2}+1) \lambda_{n+1}^{-1/2} 
	\nonumber \\
	&& \qquad \qquad+ 4 \bar{b}^{-\frac{1}{2}} (M_0+\bar{\eta}\rho_0)
	\frac{1}{\lambda_{n+1}^{1/2}}e^{-\lambda_{n+1} \chi }
	\\
	&& + 2 \left[ \beta_1 l \lambda_n^{-1} + \beta_2 \bar{b}^{-\frac{1}{2}} (M_1+\Pi\bar{\eta}) (1+l)\frac{(\gamma 
		\nu_i^{-\frac{1}{2}} + 1)}{ \lambda_{n+1}^{\frac{1}{2}}} \right] 
	\sup_{0\leq j\leq
		N-1} \tau_{N-j-1} . \nonumber
	\end{eqnarray}
	Moreover, from \eqref{teo2.2} yields that $\tau_N
	\rightarrow 0$ as $N \rightarrow \infty$, hence from \eqref{teo2.3}
	we obtain \eqref{teo2.4} for $N\rightarrow \infty$.
	To complete the proof we determine $\lambda_n$ in such a way that the 
	previous estimates are satisfied. 
	Choosing $\lambda_n \geq \delta_2$, we can write
	$l=6 \left(\frac{1}{2}+  \sup_n \left(\frac{\nu_i\lambda_{n+1}}{\lambda_{n}}\right)^{\frac{1}{2}} \right)$, as given in \eqref{l_lip}. 
	In this way, if $\lambda_n \geq \max\left(\delta_2,\delta_3\right)$ with
	\begin{equation}
	\delta_3 \geq 4 \frac{\left(M_1+\Pi \bar{\eta}\right)^2}{\bar{b}}  
	\left[ 
	\left(3+6 \sup_n \left(\frac{ \nu_i\lambda_{n+1}}{  
		\lambda_{n}}\right)^{\frac{1}{2}} \right) \nu_i^{-\frac{1}{2}} 
	+ \left(\gamma \nu_i^{-\frac{1}{2}}+1\right)
	\right]^2,
	\end{equation}
	then  condition $\mu \leq \frac{1}{2}$ is satisfied and the proof is complete.  
\end{proof}

\section{Unbounded domain: asymptotic $L^2$ decay}

In this section we consider $\Omega=\mathbb{R}^2$, we suppose that  the force term $\f$ is time dependent and $\f \in L^1\left( \left[0 \right. ,  \left. +\infty \right), L^2_b(\mathbb{R}^2)  \right)$. Now the equations under consideration are as follows

\begin{eqnarray}
&&\frac{\partial \bu}{\partial t} +\bu \cdot \gr\bu +\gr p +\eta \bu = b^{-1}\gr \cdot [\nu b (\gr\bu 
+(\gr\bu)^{T}-\mathbf{I}\gr \cdot\bu)] + \mathbf{f}, \label{equation_1} \\ 
&&\gr\cdot (b \bu) = 0,\label{equation_2}  \\ 
&&\bu(\x,t=0)= \bu_0 ,\label{equation_3}
\end{eqnarray}
where $\x \in \mathbb{R}^2$, $\nu$ is the viscosity and we denote by $0<\nu_i=\inf_{\mathbb{R}^2} \nu$,  
$\eta$ is a smooth strictly positive function and $b(\x)$ represents the bottom topography of the basin satisfying 
\begin{equation}
0 < b_i \leq b(\mathbf{x}) \leq b_s.  \nonumber
\end{equation}

The Fourier splitting method, will be used to establish  the asymptotic $L^2$-decay of the weak solutions to the shallow water model with varying bottom topography.

\subsection{Mathematical settings}\label{sec:unbounded}

We denote by $L^2_b(\mathbb{R}^2)$ the weighted $L^2(\mathbb{R}^2)$ space with scalar product and norm defined by
\begin{equation}
\left( \bu,\bv \right)_b=\int_{\mathbb{R}^2} b \bu \cdot \bv d \x, \qquad \left\| \bu \right\|_{b,2}^2=\int_{\mathbb{R}^2} b \left|\bu \right|^2 d \x, \nonumber 
\end{equation}
and  $\left\| \cdot \right\|_p$ will denote the usual norm in $L^p(\mathbb{R}^2)$.
We also use the following notation for our spaces 
\begin{eqnarray}
H &=& \left\{ \bu :  \bu \in L^2_b \left(\mathbb{R}^2\right),  \quad \gr\cdot (b \bu)=0 \right\}, \\
V &=& \left\{ \bu :  \bu \in H^1_b \left(\mathbb{R}^2\right), \quad \gr\cdot (b \bu)=0  \right\},
\end{eqnarray}
and 
\begin{equation}
V_o = \left\{ \bu : \quad \bu \in H^1_b \left(\mathbb{R}^2\right) \cap \mathbb{S}^{'}\left(\mathbb{R}^2\right), \quad \gr\cdot (b \bu)=0  \right\},
\end{equation}
where $\mathbb{S}\left(\mathbb{R}^2\right)$ is the Schwartz class of smooth, rapidly decreasing functions.

A function $\bu(\x,t) \in C_w \left( \left[0 \right. , \left. \infty \right), H  \right)$ if $\bu \in L^{\infty} \left( \left[0 \right. , \left. \infty \right), H  \right)$ and $\left( \bu, \p  \right)_b$ is continuous with respect to time $t\geq 0$, for all $\p \in H^{'}$. 

As usual, the Fourier transform of an integral function $\bv (\x) \in L^2\left( \mathbb{R}^2\right)$ is $\hat{\bv}(\xi)=\int_{\mathbb{R}^2} \bv(\x) e^{-i \x \cdot \xi} d \x$.

\noindent A weak solution $\bu$ of problem \eqref{equation_1}-\eqref{equation_3} is a function belonging to $C_w \left( \left[0 , T \right], H  \right) \cap L^{2}\left( \left[0 , T \right], V_o  \right) $ for each $T>0$, satisfying the integral relation
\begin{eqnarray*}
	&&\left( \bu(t), \p(t) \right)_b + \int_0^t \left\{ - \left( \bu, \frac{\partial \p}{\partial t} \right)_b + \right.\\
	&& \qquad \left.+ \nu \left( \left(\gr\bu 
	+(\gr\bu)^{T}-\mathbf{I}\gr \cdot\bu \right)
	: \left(\gr \p 
	+(\gr\p)^{T}-\mathbf{I}\gr \cdot\p \right) \right)_b +\right.\\
	&& \qquad \qquad\left.+ \left( \eta \bu, \p\right)_b+ \left( \bu \cdot \gr\bu , \p \right)_b \right\} d\tau  =\int_0^t \left( \p,\f \right)_b d \tau + \left(u_0, \p(0)  \right)_b,
\end{eqnarray*}
for all $t\geq s\geq0$ and for every smooth vector fields 
$$\p \in C \left( \left[0 \right. ,  \left. +\infty \right), V  \right) \cap C^1 \left( \left[0 \right. ,  \left. +\infty \right), H  \right).$$

It is easy to prove the following two Propositions where, respectively, the strong energy inequality and the generalized energy inequality are given for a weak solution of \eqref{equation_1}-\eqref{equation_3}.

\begin{proposition}
	Let $\bu_0 \in L^2(\mathbb{R}^2)$ and $\f \in L^1\left( \left[0 \right. ,  \left. +\infty \right), L^2_b(\mathbb{R}^2)  \right)$. Then, for every $T>0$, there exists a unique weak solution $\bu(\x,t) \in C_w \left( \left[0 , T \right], H  \right) \cap L^{2}\left( \left[0 , T \right], V_o  \right) $ of system \eqref{equation_1}-\eqref{equation_3}, which satisfies the following strong energy inequality:
	\begin{equation}\label{diseq_energia}
	b_i \left\| \bu(t) \right\|_2^2+2 \nu_i b_i \int_s^t \left\| \gr \bu(\tau) \right\|_{2}^2 d \tau \leq b_s \left\| \bu(s) \right\|_2^2 + 2\int_s^t \left( \bu, \f \right)_{b} d \tau,  
	\end{equation}
	for almost all $s\geq 0$ including $s=0$ and all $t\geq s \geq 0$. 
\end{proposition}

\begin{proof}
	The existence and uniqness of a weak solution to problem \eqref{equation_1}-\eqref{equation_3} satisfying the strong energy inequality \eqref{diseq_energia} follows by an application of the  standard Galerkin technique (see ~\cite{Te3}).  
	  \end{proof}

Let $\bu(\x,t)= (u_1(\x,t), u_2(\x,t))$ be a vector function  and $\psi(\x,t)$ be a scalar function. In the sequel we use the notation 
\begin{equation} \psi^{'}=\partial_t \psi,\;\;
\psi \ast \bu =\left( \psi \ast u_1, \psi \ast u_2 \right), 
\end{equation}
where the convolution is calculated with respect to the $\x$ variable.

\begin{proposition}
	Let $\bu_0 \in L^2(\mathbb{R}^2)$ and $\f \in L^1\left( \left[0 \right. ,  \left. +\infty \right), L^2_b(\mathbb{R}^2)  \right)$. Let $Z\in C^1\left[0\right.,\left.\infty\right)$ with $Z(t)\geq 0$, and 
	$\psi(t)\in C^1 \left( \left[ 0 \right. ,\left. \infty \right); \mathcal{S}\left(\mathbb{R}^2\right) \right)$ 
	be arbitrary functions. 
	Let $\bu$ be a weak solution of system \eqref{equation_1}-\eqref{equation_3}, then the following generalized energy inequality holds:
	\begin{eqnarray}\label{energy}
	&& Z(t) b_i \left\| \psi(t) \ast \bu (t) \right\|_{2}^{2} \leq b_s Z(s) \left\| \psi(s) \ast \bu(s) \right\|_{2}^{2} \nonumber\\ 
	&&\qquad\qquad\qquad\qquad\qquad + b_s \int_{s}^{t} Z^{'}(\tau) \left\| \psi(\tau) \ast \bu (\tau) \right\|_{2}^{2} d \tau \nonumber \\
	&& \qquad\qquad + 2 \int_{s}^{t} Z(\tau) \left( \psi^{'}(\tau) \ast \bu (\tau), \psi(\tau) \ast \bu (\tau) \right)_b d \tau \\
	&&\qquad\qquad\qquad\qquad\qquad - 2\nu_i b_i \int_{s}^{t} Z(\tau) \left\| \psi(\tau) \ast \gr \bu (\tau) \right\|_{2}^{2} d \tau \nonumber \\
	&& + 2 \int_{s}^{t} Z(\tau) \left[ \left(\bu\cdot \gr \bu,\psi\ast \psi \ast \bu\right)_b(\tau)\right] d \tau
	+ 2\int_s^t Z(\tau) \left( \psi \ast \bu, \f \right)_{b} d \tau . \nonumber
	\end{eqnarray}
	for almost all $s\geq 0$ including $s=0$ and all $t\geq s \geq 0$.
\end{proposition}

\begin{proof}
	To prove the generalized energy inequality \eqref{energy} one can follow ~\cite{KPSark}~\cite{ORS97}.   
	  \end{proof}

\vskip0.25cm
We give two preliminary Lemmas which are consequence of the generalized energy inequality \eqref{energy}.

\begin{lemma}
	Let $\bu$ be a weak solution of \eqref{equation_1}-\eqref{equation_3} satisfying the generalized energy inequality \eqref{energy} of Lemma 1. Then for every $\varphi \in \mathcal{S}(\mathbb{R}^2)$, we have: 
	\begin{eqnarray}\label{corollario_1}
	b_i \left\| \varphi \ast \bu (t) \right\|_{2}^{2}  &\leq& b_s \left\| e^{\frac{\nu_i b_i }{b}(t-s)\Delta} \varphi \ast \bu(s) \right\|_{2}^{2} \nonumber \\
	&&+ 2 \int_{s}^{t} \left[ \left(\bu\cdot \gr \bu,e^{2\frac{\nu_i b_i }{b}(t-\tau)\Delta}(\varphi\ast \varphi) \ast \bu\right)_b(\tau)\right] d \tau \nonumber \\
	&& +2 \int_s^t \left(  e^{\frac{\nu_i b_i }{b}(t-\tau)\Delta} \varphi \ast \bu(s), \f \right)_b d \tau ,
	\end{eqnarray}
	for almost all $s\geq 0$ including $s=0$ and all $t\geq s \geq 0$.
\end{lemma}

\begin{proof}
	Apply \eqref{energy} with $Z(t)=1$ and $\psi=e^{\frac{\nu_i b_i }{b}(t+\delta-\tau)\Delta} \varphi $ and let $\delta \rightarrow 0$ (see ~\cite{KPSark}~\cite{ORS97}).  
	  \end{proof}

\begin{lemma}
	Let $Z(t)\in C^1\left[0\right.,\left.+\infty\right)$ with $Z(t)\geq 0$. Let $\bu$ be a weak solution of \eqref{equation_1}-\eqref{equation_3} satisfying the generalized energy inequality \eqref{energy} of Lemma 2.
	Then for every $\varphi \in \mathcal{S}(\mathbb{R}^2)$, we have: 
	\begin{eqnarray}\label{corollario_2}
	&&\hskip0.8cm Z(t) b_i \left\| \bu(t) -\varphi \ast \bu (t) \right\|_{2}^{2}  \leq b_s Z(t) \left\| \bu(s)-\varphi \ast \bu(s) \right\|_{2}^{2} \\
	&&\hskip-0.7cm + b_s \int_{s}^{t} Z^{'}(\tau) \left\| \bu(\tau)-\varphi \ast \bu (\tau) \right\|_{2}^{2} d \tau - 2 b_i \nu_i  \int_{s}^{t} Z(\tau) \left\| \gr \bu(\tau)-\varphi \ast \gr \bu (\tau) \right\|_{2}^{2} d \tau \nonumber \\
	&&+ 2 \int_{s}^{t} Z(\tau) \left[ \left(\bu \cdot \gr \bu,\varphi\ast \varphi \ast \bu-2\varphi \ast \bu\right)_b(\tau)\right] d \tau + 2\int_s^t \left( \bu - \varphi \ast \bu, \f \right)_b. \nonumber
	\end{eqnarray} 
	for almost $s\geq 0$ including $s=0$ and all $t\geq s \geq 0$.
\end{lemma}

\begin{proof}
	Apply \eqref{energy} with $\psi=\zeta_n-\varphi$, with $\zeta_n(\x)=n^{-1} \zeta(\x / n)$ is a smooth and compactly supported approximation of the Dirac measure, and let $n \rightarrow \infty$ (see ~\cite{KPSark}~\cite{ORS97}). 
	  \end{proof}

\subsection{Non-Uniform decay}\label{sec:nu-decay}

We now state the main theorem of the section:

\begin{theorem}
	Let $\bu_0 \in L^2(\mathbb{R}^2)$ and $\f \in L^1\left( \left[0 \right. ,  \left. +\infty \right), L^2_b(\mathbb{R}^2)  \right)$. Let $\bu$ be a weak solution of problem \eqref{equation_1}-\eqref{equation_3},  then
	\begin{equation}
	\lim_{t \rightarrow +\infty} \left\| \bu \right\|_2 =0. 
	\end{equation}
	
\end{theorem}

\begin{proof} The proof is based on ideas of 
	~\cite{KPSark},~\cite{ORS97}:\\
	We decompose the $L^2$-norm of the Fourier transform of the weak solution $\bu$ as follows
	\begin{equation}\label{inizio}
	\left\| \bu(t) \right\|_2 = \left\| \hat{\bu}(t) \right\|_2 \leq \left\| \check{\varphi} \hat{\bu}(t) \right\|_2 
	+  \left\| \left( 1- \check{\varphi} \right) \hat{\bu}(t) \right\|_2, 
	\end{equation}
	where $\displaystyle{\check{\varphi}(\boldmath{\xi})=e^{-\left|\boldmath{\xi}\right|^2}}$ is the inverse Fourier Transform of $\displaystyle{\varphi(\x)=\frac{1}{4\pi}e^{-\frac{\left| \x \right|^2}{4}}}$, the fundamental solution of the heat equation at $t=1$.
	We estimate separately the low frequencies  and the high energy frequencies  terms in \eqref{inizio}.
	
	\medskip 
	
	\noindent {\it Low frequencies term estimate}: Using Plancherel identity and \eqref{corollario_1}, we have 
	\begin{eqnarray*}
		b_i \left\| \check{\varphi}\hat{\bu}(t) \right\|_2^2  &=&b_i \left\| \varphi \ast \bu(t) \right\|_2^2 \leq b_s \left\| e^{\frac{\nu_i b_i }{b}(t-s)\Delta} \varphi \ast \bu(s) \right\|_2^2 \\
		&&  + 2 \int_s^t \left| \left(\bu \cdot \gr \bu,e^{2\frac{\nu_i b_i }{b}(t-\tau)\Delta}\varphi \ast \varphi \ast \bu \right)_b(\tau) \right|d \tau \\
		&&\qquad+ 2\int_s^t \left(  e^{\frac{\nu_i b_i }{b}(t-\tau)\Delta} \varphi \ast \bu(s), \f \right)_b d \tau.
	\end{eqnarray*}
	
	Using (see ~\cite{Te3}) the Schwarz, H\"{o}lder and Young inequalities  and the Gagliardo-Nirenberg interpolation inequality, we have
	\begin{eqnarray*}
		\left| \left(\bu \cdot \gr \bu,e^{2\frac{\nu_i b_i }{b}(t-\tau)\Delta}\varphi \ast \varphi \ast \bu \right)_b(\tau) \right| 
		\leq C \left\| \bu\right\|_4 \left\| \gr \bu \right\|_2 \left\| e^{2\frac{\nu_i b_i }{b}(t-s)\Delta} \varphi \ast \varphi \ast \bu \right\|_4 && \\
		\leq C \left\| e^{2\frac{\nu_i b_i }{b}(t-s)\Delta} \varphi \ast \varphi \right\|_{1} \left\| \gr \bu \right\|_{2} \left\| \bu \right\|_4^2 &&\\
		\leq C \left\| e^{2\frac{\nu_i b_i }{b}(t-s)\Delta} \varphi \ast \varphi \right\|_{1} \left\| \gr \bu \right\|_2^2 \left\| \bu \right\|_2 .&&
	\end{eqnarray*}

	It is easy to prove (see ~\cite{W87}) that there exists a constant $\kappa=\kappa(\bu_0,\f)$ such that
	\begin{equation}
	\left| \left( \bu, \f \right)_b \right| \leq  \kappa \left\| \f \right\|_{b,2}, \quad \mathrm{and} \qquad \left| \left( e^{\frac{\nu_i b_i }{b}(t-\tau)\Delta} \varphi \ast \bu, \f \right)_b \right| \leq  \kappa \left\| \f \right\|_{b,2}.
	\end{equation}
	From the strong energy inequality \eqref{diseq_energia} we have that
	
	\begin{equation}\label{diseq_energia_bis}
	\left\| \bu (t) \right\|_2^2 \leq \left\| \bu_0  \right\|_2^2 + 2\kappa \int_0^t \left\| \f \right\|_{b,2} d \tau.   
	\end{equation}
	
	Hence
	\begin{eqnarray}
	\left\| \check{\varphi}\hat{\bu}(t) \right\|_2^2 &\leq& \frac{b_s}{b_i} \left\| e^{\frac{\nu_i b_i }{b}(t-s)\Delta} \varphi \ast \bu(s) \right\|_2^2 \nonumber \\ 
	&& + \frac{C}{b_i} \left(\left\| \bu_0 \right\|_2^2 +2 \kappa \int_0^{+\infty} \left\| \f \right\|_{b,2} d \tau \right)^{\frac{1}{2}} \int_s^{+\infty} \left\| \gr \bu \right\|_2^2 d \tau \nonumber \\
	&&+2\frac{\kappa}{b_i} \int_s^{+\infty} \left\| \f \right\|_{b,2} d \tau. \nonumber
	\end{eqnarray}
	
	By the Lebesgue dominated convergence theorem, it follows that, as $t\rightarrow +\infty$,
	\begin{eqnarray}
	\left\| e^{\frac{\nu_i b_i }{b}(t-s)\Delta} \varphi \ast \bu(s) \right\|_2^2&\leq& \left\| e^{\nu_i (t-s)\Delta} \varphi \ast \bu(s) \right\|_2^2 = \nonumber \\
	&&=\left\| e^{-\nu_i (t-s)\boldmath{\xi}^2} \check{\varphi} \hat{\bu}(s) \right\|_2^2 \rightarrow 0,
	\end{eqnarray}
	for each $s\geq0$, since $\check\varphi \hat{\bu}(s) \in L^2(\mathbb{R}^2)$.
	
	Since $\displaystyle{\int_0^{+\infty} \left\| \gr \bu \right\|_2^2 d \tau < \infty}$ by the strong energy inequality \eqref{diseq_energia} and\\ $\displaystyle{\int_0^{+\infty} \left\| \f \right\|_{b,2} d \tau < \infty}$ by hypothesis, the quantities $\displaystyle{\int_s^{+\infty} \left\| \gr \bu \right\|_2^2 d \tau }$ and \\$\displaystyle{\int_s^{+\infty} \left\| \f \right\|_{b,2} d \tau }$ are small for $s$ suitable large, then  $\displaystyle{\left\| \check{\varphi}\hat{\bu}(t) \right\|_2 \rightarrow 0}$ as $\displaystyle{t \rightarrow 0}$. 
	
	\medskip
	\vskip0.25cm
	\noindent {\it High frequencies term estimate:} Use Corollary \eqref{corollario_2} with $\displaystyle{\check{\varphi}(\boldmath{\xi})=e^{-\left| \boldmath{\xi} \right|^2}}$, and $Z(t)$ determined below. 
	Consider a function $G(t)\geq 0$, to be determined below, and apply the Fourier splitting method to the first two terms in \eqref{corollario_2}: 
	\begin{eqnarray*}
		&&\int_{s}^{t} Z^{'}(\tau) \left\| \bu(\tau)-\varphi \ast \bu (\tau) \right\|_{2}^{2} d \tau - 2 b_s \int_{s}^{t} Z(\tau) \left\| \gr \bu(\tau)-\varphi \ast \gr \bu (\tau) \right\|_{2}^{2} d \tau \\
		&&\qquad\qquad\qquad= \int_{s}^{t} Z^{'}(\tau) \int_{\left|\boldmath{\xi}\right|>G} 
		\left| \left( 1- \check{\varphi}(\boldmath{\xi}) \right) \hat{\bu}(\boldmath{\xi},\tau) \right|^{2} 
		d \boldmath{\xi} d \tau \\ 
		&&\qquad\qquad\qquad- 2 \int_{s}^{t} Z(\tau) \int_{\left|\boldmath{\xi}\right|>G} b_s \left| \left|\boldmath{\xi}\right| \left(1-\check{\varphi}(\boldmath{\xi}) \right) \hat{\bu}(\boldmath{\xi},\tau)\right|^{2} d \boldmath{\xi} d \tau \\
		&&\qquad\qquad\qquad+ \int_{s}^{t} Z^{'}(\tau) \int_{\left|\boldmath{\xi}\right|\leq G} 
		\left| \left( 1- \check{\varphi}(\boldmath{\xi}) \right) \hat{\bu}(\boldmath{\xi},\tau) \right|^{2} 
		d \boldmath{\xi} d \tau \\ 
		&&\qquad\qquad\qquad- 2 \int_{s}^{t} Z(\tau) \int_{\left|\boldmath{\xi}\right| \leq G} b_s \left| \left|\boldmath{\xi}\right| \left(1-\check{\varphi}(\boldmath{\xi}) \right) \hat{\bu}(\boldmath{\xi},\tau)\right|^{2} d \boldmath{\xi} d \tau. 
	\end{eqnarray*}
	Choose
	\begin{equation}
	Z(t)=(1+t)^\alpha \qquad \mathrm{and} \qquad G^2=\frac{\alpha}{2 b_s (t+1)}, 
	\end{equation}
	with $\alpha>0$ fixed, then $Z(t)$ and $G(t)$ satisfies the following equation:
	\begin{equation}
	Z^{'}(t)-2 b_s Z(t)G^2(t)=0. \nonumber 
	\end{equation}
	Hence the last equation is reduced to 
	\begin{eqnarray*}
		&&\int_{s}^{t} Z^{'}(\tau) \int_{\left|\boldmath{\xi}\right|>G} 
		\left| \left( 1- \check{\varphi}(\boldmath{\xi}) \right) \hat{\bu}(\boldmath{\xi},\tau) \right|^{2} 
		d \boldmath{\xi} d \tau \\ 
		&&\qquad\qquad- 2 \int_{s}^{t} Z(\tau) \int_{\left|\boldmath{\xi}\right|>G} b_s \left| \left|\boldmath{\xi}\right| \left(1-\check{\varphi}(\boldmath{\xi}) \right) \hat{\bu}(\boldmath{\xi},\tau)\right|^{2} d \boldmath{\xi} d \tau \\
		&& \leq \int_s^t \left( Z^{'}-2b_sZG^2 \right) \int_{\boldmath{\left| \xi \right|>0}} 
		\left| \left(1-\check{\varphi}(\boldmath{\xi}) \right) \hat{\bu}(\boldmath{\xi},\tau)\right|^{2} d \boldmath{\xi} d \tau =0 . 
	\end{eqnarray*}
	
	As $\left| 1- \check{\varphi}(\boldmath{\xi}) \right| \leq \left| \boldmath{\xi} \right|^2$, then for small $\left| \boldmath{\xi} \right|$ we have 
	
	\begin{eqnarray*}
		&&\int_{s}^{t} Z^{'}(\tau) \int_{\left|\boldmath{\xi}\right|\leq G} 
		\left| \left( 1- \check{\varphi}(\boldmath{\xi}) \right) \hat{\bu}(\boldmath{\xi},\tau) \right|^{2} 
		d \boldmath{\xi} d \tau \leq \\
		&&\qquad\leq C \left\| \bu_0 \right\| \int_s^t Z^{'}(\tau) G^4(\tau) d \tau \leq  C \int_s^t \left(1 +\tau \right)^{\alpha-3} d \tau.
	\end{eqnarray*}
	
	The last two terms in \eqref{corollario_2} can be simplified denoting by $\chi=\varphi \ast \varphi-2 \varphi$, and combining (see ~\cite{Te3})the  Schwarz, the H\"{o}lder and the Young inequalities, the Gagliardo-Nirenberg interpolation inequality, and the strong energy inequality \eqref{diseq_energia_bis}, 
	
	\begin{eqnarray*}
		& &\int_s^t Z(\tau) \left| \left(\bu \cdot \gr \bu, \varphi \ast \varphi \ast \bu - 2 \varphi \ast \bu \right)_b (\tau) \right| d \tau = \\
		&& \quad = \int_s^t Z(\tau) \left| \left(\bu \cdot \gr \bu, \chi \ast \bu \right)_b (\tau) \right| d \tau \leq\int_s^t Z(\tau) \left\| \bu \right\|_4 \left\| \gr \bu \right\|_{2} \left\| \chi \ast \bu \right\|_4 d \tau \\
		& &\quad \leq C \left\| \chi \right\|_{1} \int_s^t Z(\tau) \left\| \bu\right\|_4^2 \left\| \gr \bu \right\|_{2} d \tau  \leq C \left\| \chi \right\|_{1} \int_s^t Z(\tau) \left\| \bu\right\|_2 \left\| \gr \bu \right\|_{2}^2 d \tau \\
		& &\quad \leq C \left\| \chi \right\|_{1} \left(\left\| \bu_0 \right\|_2^2 +2 \kappa \int_0^{+\infty} \left\| \f \right\|_{b,2} d \tau \right)^{\frac{1}{2}} \int_s^t Z(\tau) \left\| \gr \bu \right\|_{2}^2 d \tau,
	\end{eqnarray*}
	and
	
	\begin{eqnarray*}
		\int_s^t Z(\tau) \left( \varphi \ast \bu , \f \right)_b d \tau \leq  \kappa \int_s^t Z(\tau) \left\| \f \right\|_{b,2} d \tau . \nonumber
	\end{eqnarray*}
	Combining the previous estimates yields
	
	\begin{eqnarray*}
		\left\| \left(1-\check{\varphi}\right) \hat{\bu} (t) \right\|_2^2 &\leq & \frac{b_s Z(s)}{b_i Z(t)} \left\| \left(1-\check{\varphi}\right) \hat{\bu} (s) \right\|_2^2 + \frac{C}{Z(t)} \int_s^t (1+\tau)^{\alpha-3} d \tau \\
		&& \quad + \frac{1}{Z(t)} \left( C \int_s^t Z(\tau) \left\| \gr \bu \right \|_2^2 d \tau + \kappa  \int_s^t Z(\tau) \left\| \f \right \|_{b,2} d \tau \right).
	\end{eqnarray*}
	We compute the $\limsup$ as $t\rightarrow +\infty$ for fixed $s>0$.
	
	\noindent Since $Z(t)=(1+t)^{\alpha}$ for  some $\alpha>0$, it follows that  $\displaystyle{\frac{Z(s)}{Z(t)} \rightarrow 0}$ when $t\rightarrow +\infty$. Moreover we have that
	
	\begin{equation}
	\limsup_{t \rightarrow +\infty} \frac{1}{(t+1)^{\alpha}} \int_s^t (1+\tau)^{\alpha-3} d \tau=0. \nonumber
	\end{equation}
	
	As $\displaystyle{\frac{Z(\tau)}{Z(t)} \leq 1}$ for $\tau \in \left[0,t\right]$, then
	
	\begin{equation}
	\limsup_{t \rightarrow +\infty} \left\| \left(1-\check{\varphi}\right) \hat{\bu} (t) \right\|_2^2 \leq C \int_s^{+\infty} \left\| \gr \bu (\tau) \right \|_2^2 d \tau + 2 \kappa \int_s^{+\infty} \left\| \f \right \|_{b,2} d \tau ,
	\end{equation}
	
	hence $\limsup_{t \rightarrow +\infty} \left\| \left(1-\check{\varphi}\right) \hat{\bu} (t) \right\|_2^2 d =0$, for $s$ sufficiently large. 
\end{proof}

\subsection{Uniform decay}\label{sec:u-decay}

In this section we want to prove the uniform rate of decay for the solutions of the viscous shallow water equations \eqref{equation_1}-\eqref{equation_3}.

We suppose for simplicity, that $\nu$ is a constant, then \eqref{equation_1}-\eqref{equation_3} can be written as:

\begin{eqnarray}
&&\frac{\partial \bu}{\partial t} +\bu \cdot \gr\bu +\eta \bu + \gr p  = \frac{\nu}{b}\gr \cdot [b (\gr\bu 
+(\gr\bu)^{T}-\mathbf{I}\gr \cdot\bu)] + \f , \label{equation_1_bis} \\ 
&&\gr\cdot (b \bu) = 0,\label{equation_2_bis}  \\ 
&&\bu(\x,t=0)= \bu_0 .\label{equation_3_bis}
\end{eqnarray}

Suppose that the force term $\f$ satisfies the following properties:
\begin{eqnarray}
&&\f =D \bold{g}, \, \text{where}\, D \, \text{is any first order derivative} \label{hypot_force_1} \\
&&\qquad\qquad\, \text{and} \, \bold{g} \in 
L^\infty \left( \left[0 \right. ,  \left. +\infty \right), L^1(\mathbb{R}^2)  \right),\nonumber\\
&&
\| \f \|_2 \leq \kappa (e+t)^{-2}.\label{hypot_force_2}
\end{eqnarray}

In particular we prove the following theorem:
\vskip0.5cm	

\begin{theorem}
	
	Suppose that $\bu_0 \in L^2(\mathbb{R}^2)\cap L^1(\mathbb{R}^2)$ and let $\bu$ be the weak solution of the viscous shallow water equations \eqref{equation_1_bis}-\eqref{equation_3_bis}.
	Suppose that $\f$ satisfies \eqref{hypot_force_1} and \eqref{hypot_force_2}, then
	
	\begin{equation}
	\left\| \bu \right\|_2 \leq C \left(\log(e+t)\right)^{-1/2}, 
	\end{equation}
	
	with $C$ a constant which depends on $\f$, $b$, $\eta$ and $\bu_0$. 		 
	
\end{theorem}

\vskip0.5cm	
\noindent Before establishing  the proof of the theorem, we give three preliminary lemmas:

\vskip0.5cm
\begin{lemma}
	($L^p-L^q$)-type estimate: Let us consider $\bu_0 \in L^q\cap L^2$, with $1\leq q <2$, then
	\begin{equation}
	\| e^{-\left[A_{b\nu}-\eta I \right] t } \bu_0\|_2 \leq C t^{-(1/q-1/2)}(\|\bu_0\|_2+\|\bu_0\|_q).
	\end{equation}	
\end{lemma}		

\begin{proof}
	The proof follows from the well-known ($L^p-L^q$) type estimate for the linear heat equation and observing that
	\begin{equation}
	(A_{b\nu}\bu,\bu)_{b} \equiv -\|\gr \bu\|_2^2 \equiv (\Delta \bu,\bu),
	\end{equation}	
	then, denoting with $\bu(t)=e^{-\left[A_{b\nu}-\eta I \right]t} \bu_0$ we have:
	\begin{eqnarray*}
		\| e^{-\left[A_{b\nu}-\eta I \right]t} \bu_0 \|_2^2 &\leq&\|\bu_0\|_2^2 + C\int_0^t (A_{b\nu} \bu, \bu)_b d \tau - C\inf|\eta|\int_0^t\|\bu\|_2^2d\tau\\
		&\leq& \|\bu_0\|_2^2 + C\int_0^t \|\gr \bu\|_2^2 d \tau \\
		&\leq& C \|e^{\Delta t}\bu_0\|_2^2.
	\end{eqnarray*}

\end{proof}

\vskip0.5cm
\begin{lemma}
	Suppose that $\bu_0 \in L^2(\mathbb{R}^2) \cap L^1 (\mathbb{R}^2)$
	and that $\f$ satisfies \eqref{hypot_force_1} and \eqref{hypot_force_2}. 
	Then
	the weak solution $\bu$ of the viscous shallow water equations \eqref{equation_1_bis}-\eqref{equation_3_bis} satisfies the following a priori estimate:
	\begin{equation}
	\int_0^t \| \bu(\tau) \|_2^4 d \tau   \leq C (e+t)^{-1},\label{lemma8}
	\end{equation}
	where $C$ is a constant which depends on $\bu_0 $, $\eta$, $\f$ and $b$. 
	
\end{lemma}

\begin{proof}
	From Lemma 6, we have that
	\begin{eqnarray}\label{prev}
	\|\bu\|_2 &\leq& C\|e^{-\left[A_{b\nu}-\eta I\right]t} \bu_0\|_2 +C \int_0^t \| e^{-\left[A_{b\nu}-\eta I\right](t-\tau)} P\left( \bu\cdot \gr \bu\right) \|_2 d\tau \nonumber \\
	& & \quad +C \int_0^t \| e^{-\left[A_{b\nu}-\eta I\right](t-\tau)} P\left( \f, \bu\right) \|_2 d\tau \nonumber \\
	&\leq& C t^{-1/2}\|\bu_0\|_1 + C \int_0^t (t-\tau)^{-1/2} \| \bu\cdot \gr \bu\|_1 d\tau +C \int_0^t (t-\tau)^{-1/2} \| \f \bu\|_1 d\tau
	\nonumber \\
	&\leq& C (e+t)^{-1/2}\|\bu_0\|_1 + C \int_0^t (t-\tau)^{-1/2} \| \bu\|_2\|\gr \bu\|_2 d\tau + \nonumber \\
	&& \qquad +  C \int_0^t (t-\tau)^{-1/2} \| \bu\|_2\|\f\|_2 d\tau
	\end{eqnarray}	
	Consider the generalized Young's inequality \cite{RS75,KM86} for convolution:

	if $f \in L^p$ and $g \in L^{q,w}$, with $1<p,q,r<\infty$ and $p^{-1}+r^{-1} = 1+q^{-1}$ then
	\begin{equation}\label{generalized_Young}
	\|f\star g\|_q  \leq C_{p,r}  \|f\|_p \|g\|_{r,w},
	\end{equation}
	where the $L^{r,w}$ is the weak $L^r$ space with norm
	$\|g\|_{r,w} +\sup_{t} (t^r\mu\{x:g(x)>t\})^{1/r}$.
	
	Now, let $q=4$ and $1+\frac{1}{q}=\frac{1}{2}+\frac{q+2}{2q}$, and applying \eqref{generalized_Young} to \eqref{prev}:	
	\begin{eqnarray*}
		&&\hskip-0.2cm\left[\int_0^t \| \bu(\tau) \|_2^q d \tau \right]^{1/q}\leq C \|\bu_0\|_1 (e+t)^{1/q-1/2} \\
		&&\hskip1.5cm+ C \left[ \int_0^t 
		\left( \|\bu(\tau)\|_2 \|\gr \bu (\tau)\|_2 \right)^{\frac{2q}{q+2}} d \tau \right]^{\frac{q+2}{2q}} \\
		&&\hskip1cm + C \left[ \int_0^t 
		\left( \|\bu(\tau)\|_2 \|\f (\tau)\|_2 \right)^{\frac{2q}{q+2}} d \tau \right]^{\frac{q+2}{2q}} \\
		&&\leq  C \|\bu_0\|_1 (e+t)^{1/q-1/2} + C \left[ \int_0^t \|\bu(\tau)\|_2^q d \tau \right]^{\frac{1}{q}} \left[\int_0^t \|\gr \bu (\tau)\|_2^2 d \tau \right]^{\frac{1}{2}}\\
		&& \quad +C \left[ \int_0^t \|\bu(\tau)\|_2^q d \tau \right]^{\frac{1}{q}} \left[\int_0^t \| \f \|_2^2 d \tau \right]^{\frac{1}{2}} \\
		&&\leq C \|\bu_0\|_1 (e+t)^{1/q-1/2} + C(1+\|\bu_0\|_2) \left[ \int_0^t \|\bu(\tau)\|_2^q d \tau \right]^{\frac{1}{q}},
	\end{eqnarray*}	
	and assuming that $C(1+\| \bu_0\|_2)\leq 1/2$, we have \eqref{lemma8}.
	  
\end{proof}

\vskip0.5cm
\begin{lemma}
	Suppose that $\bu_0 \in L^2(\mathbb{R}^2) \cap L^1 (\mathbb{R}^2)$
	and that $\f$ satisfies \eqref{hypot_force_1} and \eqref{hypot_force_2}. 
	Then
	the weak solution $\bu$ of the viscous shallow water equations \eqref{equation_1_bis}-\eqref{equation_3_bis} satisfies the following a priori estimate:
	\begin{eqnarray}
	|\hat{\bu}(\xi,t)|  &\leq& \|\bu_0\|_1 +C|\xi| t+ C(1+|\xi|)\int_0^t \| \bu(\tau) \|_2 d \tau \nonumber \\
	&&+ C(1+|\xi|)\int_0^t \| \bu(\tau) \|_2^2 d \tau.\label{A1}
	\end{eqnarray}
	
\end{lemma}		

\begin{proof}
	
	Write the viscous shallow water equations in the following way.
	\begin{eqnarray*}
		\frac{\partial \bu}{\partial t} &=&\nu \Delta \bu + G(\bu) +\f, \\ 
		&&\gr \cdot (b \bu)=0,
	\end{eqnarray*}
	where
	\begin{equation}
	G(\bu)=-\bu\cdot \gr \bu -\eta \bu -\gr p + \nu \frac{\gr b}{b}\left(\gr \bu +(\gr \bu)^T - \gr\cdot \bu \mathbb{I} \right). \nonumber
	\end{equation} 
	Hence
	\begin{equation}
	\hat{\bu}=e^{-|\xi|^2 t} \hat{\bu}_0 + \int_0^t e^{-|\xi|^2 (t-s)} \left( \widehat{P(G)} + \widehat{P\f}\right) d s,\label{inter}
	\end{equation}
	where $P$ is the projection form $L^2_b$ in $H$.
	
	By assumption $\f=D \boldmath{g} $ where $D$ is any first order derivative and $ \boldmath{g} \in L^\infty \left( \left[0 \right. ,  \left. +\infty \right), L^1(\mathbb{R}^2)  \right)$, hence
	$$|\widehat{P\f}|\leq C|\xi|.$$ 
	
	We prove later that
	\begin{equation}\label{A2}
	|\widehat{P(G)}|\leq C(1+ |\xi|)(\|\bu(t) \|_2^2+\|\bu(t) \|_2),
	\end{equation}
	where $C$ is a constant which depends on $b$. 
	Using \eqref{A2} in \eqref{inter} and integrate in time \eqref{inter}, we have
	\begin{equation}
	|\hat{\bu}(\xi,t)|  \leq |\hat{\bu}(\xi,0)|+ C|\xi| t +C(1+|\xi|)\int_{0}^{t} \| \bu(\tau)\|_2 d \tau + C(1+|\xi|)\int_{0}^{t} \| \bu(\tau)\|_2^2 d \tau, \nonumber
	\end{equation}
	which is \eqref{A1}.
	To complete the proof, we finally show that \eqref{A2} holds,  
	$$|\widehat{P(G)}|\leq C(1+ |\xi|)(\|\bu(t)\|_2^2+\|\bu(t) \|_2).$$
	
	Using \eqref{equation_2_bis}
	\begin{eqnarray*}
		|\widehat{P(\bu\cdot \gr \bu)}|&\leq& \left|\int \gr \cdot (\bu \otimes \bu) e^{i \xi\cdot \x} d \x \right| + \left| \int \bu \frac{\gr b}{b} \cdot \bu e^{i \xi\cdot \x} d \x \right| \\
		&\leq& |\xi| \| \bu \otimes \bu \|_1 + \| \bu \frac{\gr b}{b} \cdot \bu\|_1 \\
		&\leq& |\xi| \| \bu \|_2^2 + C \| \bu \|_2^2\leq C(1+ |\xi|)\|\bu(t)\|_2^2,
	\end{eqnarray*}

	and
	
	\begin{eqnarray*}
		&&\left|\widehat{P\left[\frac{\gr b}{b}\left(\gr \bu +(\gr \bu)^T - \gr\cdot \bu I \right)\right]}\right| \leq \left| \int \gr \left(\frac{\gr b}{b}\cdot \bu\right) e^{i \xi \cdot \x} d \x \right| + \\
		&&\qquad + 2\left| \int \gr \left(\frac{\gr b}{b}\right)\bu e^{i \xi \cdot \x} d \x \right|
		+\left| \int \gr \cdot \left(\frac{\gr b}{b}\otimes \bu\right) e^{i \xi \cdot \x} d \x \right| +\\
		&&+\left| \int \left(\gr \cdot \frac{\gr b}{b}\right)\bu e^{i \xi \cdot \x} d \x \right|+\left| \int \gr \cdot \left(\bu \otimes \frac{\gr b}{b}\right) e^{i \xi \cdot \x} d \x \right|\\
		&\leq& |\xi|\left(\|\frac{\gr b}{b}\cdot \bu \|_1 
		+\| \frac{\gr b}{b} \otimes \bu \|_1
		+\| \bu \otimes \frac{\gr b}{b} \|_1 \right) +\\
		&& + \|\left(\gr \cdot \frac{\gr b}{b}\right)\bu \|_1 
		+\| \gr \left( \frac{\gr b}{b}\right)\bu\|_1\\				
		&\leq& C(1+|\xi|)\|\bu(t)\|_2.
	\end{eqnarray*}
	
	Finally		
	\begin{equation}
	|\widehat{P(\eta \bu)}|=\left|\int \eta \bu e^{i \xi\cdot \x} d \x \right|\leq \|\eta \|_2 \| \bu(t) \|_2 \leq C\|\bu(t)\|_2 .\nonumber
	\end{equation}

\end{proof}

\vskip0.5cm

We are now in the position to give the proof of the main theorem of this section:	

\vskip0.5cm	
\begin{proof}
	Taking the scalar product of \eqref{equation_1_bis} with $\bu$ and using Plancherel's theorem, we have
	
	\begin{equation}
	\frac{d}{d t} \int_{\mathbb{R}^2} |\hat{\bu}(\xi,t)|^2 d \xi + \int_{\mathbb{R}^2} |\xi|^2 |\hat{\bu}(\xi,t)|^2 d \xi \leq  |(\f,\bu)_b|. \nonumber
	\end{equation}
	For the second term
	
	\begin{eqnarray*}
		\int_{\mathbb{R}^2} |\xi|^2 |\hat{\bu}(\xi,t)|^2 d \xi &\geq& \int_{G(t)^c} |\xi|^2 |\hat{\bu}(\xi,t)|^2 d \xi \\
		&\geq& g^2(t) \int_{G(t)^c} |\hat{\bu}(\xi,t)|^2 d \xi \\
		&=& g^2(t) \int_{\mathbb{R}^2} |\hat{\bu}(\xi,t)|^2 d \xi - g^2(t) \int_{G(t)} |\hat{\bu}(\xi,t)|^2 d \xi 
	\end{eqnarray*}
	where $G(t)=\left\{ \xi \in \mathbb{R}^2 : |\xi| < g(t) \right\}$ and $g\in C\left(\left[0,\infty \right]; \mathbb{R}_{+}\right)$ which can be determinate later.
	
	Then 
	
	\begin{equation}
	\frac{d}{d t} \int_{\mathbb{R}^2} |\hat{\bu}(\xi,t)|^2 d \xi + g^2(t) \int_{\mathbb{R}^2} |\hat{\bu}(\xi,t)|^2 d \xi \leq g^2(t) \int_{G(t)} |\hat{\bu}(\xi,t)|^2 d \xi + (\f,\bu)_b, \nonumber
	\end{equation}
	and by Lemma 7 we have
	
	\begin{eqnarray*}
		& &\frac{d}{d t} \int_{\mathbb{R}^2} |\hat{\bu}(\xi,t)|^2 d \xi + g(t)^2 \int_{\mathbb{R}^2} |\hat{\bu}(\xi,t)|^2 d \xi \leq \\
		& & \leq 2 \pi g^2(t) \int_0^{g(t)} \left[\|\bu_0\|_1 + C r t+ C(1+r)\int_0^t \| \bu(\tau) \|_2 d \tau \right.\\
		&&\qquad\qquad \left.+ C(1+r)\int_0^t \| \bu(\tau) \|_2^2 d \tau \right]^{2} r d r   + |(\f,\bu)_b|,
	\end{eqnarray*}
	and by H\"older inequality it is possible to write as:
	\begin{eqnarray*}
		& &\frac{d}{d t} \int_{\mathbb{R}^2} |\hat{\bu}(\xi,t)|^2 d \xi + g(t)^2 \int_{\mathbb{R}^2} |\hat{\bu}(\xi,t)|^2 d \xi \leq \\
		& & \leq 2 \pi g^2(t) \int_0^{g(t)} \left[\|\bu_0\|_1^2 + C r^2 t^2+ C(1+r^2) t^{\frac{3}{2}}  \left[\int_0^t \| \bu(\tau) \|_2^4 d \tau \right]^{\frac{1}{2}} +\right.\\ &&  \qquad \left.+ C(1+r^2) t \int_0^t \| \bu(\tau) \|_2^4 d \tau \right] r d r  +|(\f,\bu)_b|.
	\end{eqnarray*}
	
	Integrating in time, we have:
	
	\begin{eqnarray*}
		& & e^{2\int_{0}^{t} g^2(s) d s} \| \bu(t) \|_2^2 \leq \| \bu_0 \|_2^2 + \\
		& & + 2 \pi \|\bu_0\|_1^2 \int_{0}^{t} e^{2\int_{0}^{t} g^2(s) d s}  g^4(s) ds \\
		&& + 2 \pi C \int_{0}^{t} e^{2\int_{0}^{t} g^2(s) d s}  g^6(s) s^2 ds \\
		& & + C \int_{0}^{t} e^{2\int_{0}^{t} g^2(s) d s} (g^4(s)+g^6(s)) s^{\frac{3}{2}} \left[\int_0^s \| \bu(\tau) \|_2^4 d \tau \right]^{\frac{1}{2}} ds \\
		& & + C \int_{0}^{t} e^{2\int_{0}^{t} g^2(s) d s} (g^4(s)+g^6(s)) s \left[\int_0^s \| \bu(\tau) \|_2^4 d \tau \right]d s \\
		&&+ \int_{0}^{t} e^{2\int_{0}^{t} g^2(s) d s}| (\f,\bu)_b |d s.
	\end{eqnarray*}
	
	To obtain a basic estimate, we take 
	
	\begin{eqnarray*}
		&& g^2(t)=\frac{1}{(e+t)\log(e+t)}, \\
		&& e^{2\int_{0}^{t} g^2(s) d s} = \left[\log(e+t)\right]^{2},
	\end{eqnarray*}
	
	then
	
	\begin{eqnarray*}
		& & \int_{0}^{t} e^{2\int_{0}^{t} g^2(s) d s}  g^4(s) ds \leq 
		C \int_{0}^{t} \frac{1}{(e+s)^2} ds  \leq C , \\
		& & \int_{0}^{t} e^{2\int_{0}^{t} g^2(s) d s}  g^6(s) s^2 ds \leq 
		C \int_{0}^{t} \frac{s^2}{(e+s)^3 \log(e+s)} ds  \leq C \log\left(\log(e+t)\right) , \\
		& & \int_{0}^{t} e^{2\int_{0}^{t} g^2(s) d s} g^6(s) s^{\frac{3}{2}} \left[\int_0^s \| \bu(\tau) \|_2^4 d \tau \right]^{\frac{1}{2}} d s \leq\\
		&& \hskip2cm \leq 
		C \int_{0}^{t} \frac{s^2 \|\bu_0 \|^2_2}{(e+s)^3 \log(e+s)} ds  \leq C \log\left(\log(e+t)\right) ,\\
		& & \int_{0}^{t} e^{2\int_{0}^{t} g^2(s) d s} g^6(s) s \left[\int_0^s \| \bu(\tau) \|_2^4 d \tau \right]d s \leq\\
		&& \hskip2cm \leq C \int_{0}^{t} \frac{s^2 \|\bu_0 \|_2^4}{(e+s)^3 \log(e+s)} ds  \leq C \log\left(\log(e+t)\right),
	\end{eqnarray*}
	
	and using Lemma 8:
	
	\begin{eqnarray*}
		& & \int_{0}^{t} e^{2\int_{ 0}^{t} g^2(s) d s} g^4(s) s^{\frac{3}{2}} \left[\int_0^s \| \bu(\tau) \|_2^4 d \tau \right]^{\frac{1}{2}} ds \leq\\
		&& \hskip3cm \leq  
		C \int_{0}^{t}  \frac{s^{\frac{3}{2}} }{(e+s)^{\frac{5}{2}}} ds  \leq C\log(e+t) ,\\
		& & \int_{0}^{t} e^{2\int_{ 0}^{t} g^2(s) d s} g^4(s) s \left[\int_0^s \| \bu(\tau) \|_2^4 d \tau \right] ds \leq 
		C \int_{0}^{t} \frac{s }{(e+s)^3} ds  \leq C .\\
	\end{eqnarray*}
	
	Finally, using the hypothesis \eqref{hypot_force_2} on $\f$ the last term is
	
	\begin{eqnarray*}
		& & \int_{0}^{t} e^{2\int_{ 0}^{t} g^2(s) d s} |(\f,\bu)_b|ds \leq C \int_{0}^{t} \left[\log(e+s)\right]^{2} \|\f\|_2 \|\bu\|_2 ds \\
		&&\qquad  \leq C \|\bu_0\|_2 \int_{0}^{t} \frac{\log^2(e+s)}{(e+s)^2}   ds \leq C 
	\end{eqnarray*}	
	
	Hence
	
	\begin{equation}
	\left[\log(e+t)\right]^{2}\|\bu(t) \|_2^2 \leq C \left[1 + \log\left(\log(e+t)\right)+ \log(e+t)\right] , \nonumber
	\end{equation}	
	
	and the theorem is proved.	  
	
\end{proof}


\section*{Acknowledgement}

The work of the authors has been partially supported by FFR grant of the Department of Mathematics University of Palermo.

The work of M.E. Schonbek has been partially supported by NSF Grant DMS-0900909.

The work of M. Sammartino has been partially supported by the GNFM of INDAM.

The work of V. Sciacca has been partially supported by
GNFM/INdAM through a Progetto Giovani grant.









\end{document}